\documentclass[10pt]{amsart}

\usepackage[utf8]{inputenc}
\usepackage{amsmath}
\usepackage{amssymb}
\usepackage{enumerate}
\usepackage{mathpazo}
\usepackage{amsthm}
\usepackage{tikz-cd}
\usepackage{mathrsfs}
\usepackage{amsrefs}
\usepackage{mathtools}
\usepackage[normalem]{ulem}
\usepackage{verbatim} %multiline commenting

\newcommand\xxoverset[3]{%
  \resizebox{#1+\widthof{\scriptsize #2}}{\height}{$#3$}}
  \newcommand\extoverset[3][0pt]{%
  \mathrel{\overset{\textup{#2}}{\xxoverset{#1}{#2}{#3}}}}

\usepackage{hyperref}
\hypersetup{
    colorlinks=true,
    linkcolor=blue,
    citecolor=cyan,
    filecolor=magenta,      
    urlcolor=cyan,
    pdftitle={Lattice Actions},
    pdfpagemode=FullScreen,
    }

\usepackage{geometry}
\usepackage{setspace,kantlipsum} %spacing

\usepackage{tikz}
\usetikzlibrary{matrix,arrows}

\usepackage{graphicx}

\usepackage{cite}

\newcommand{\N}{\mathbb{N}}
\newcommand{\Z}{\mathbb{Z}}
\newcommand{\R}{\mathbb{R}}

\newcommand{\Aff}{\mathrm{Aff}}

\theoremstyle{definition}
\newtheorem{thm}{Theorem}[section]
\newtheorem*{thm*}{Theorem}
\newtheorem{defn}[thm]{Definition}
\newtheorem{prop}[thm]{Proposition}
\newtheorem{cor}[thm]{Corollary}

\newtheorem{lem}[thm]{Lemma}
\newtheorem{remark}[thm]{Remark}

\newtheorem{ex}[thm]{Example}

\newtheorem{cl}[thm]{Claim}

\title{Cylinder decompositions on geometric armadillo tails}
\author{Dami Lee}
\address{Department of Mathematics, Oklahoma State University, Stillwater, OK 74078, USA}
\email{dami.lee@okstate.edu}
\author{Josh Southerland}
\address{School of Mathematics, University of Bristol, Woodland Rd, Bristol BS8 1UG, UK}
\email{josh.southerland@bristol.ac.uk}
\date{}
\allowdisplaybreaks

\begin{document}

\begin{abstract}
We study a class of finite-area, infinite-type translation surfaces and construct an explicit cylinder decomposition on these surfaces. The cylinder decomposition contains a special curve which distinguishes it from cylinder decompositions in other periodic directions. 
\end{abstract}

\setstretch{1.25}


\maketitle

%\tableofcontents

\section{Introduction}\label{sec:intro}

A translation surface is a (countable) collection of polygons in the plane where all edges are paired with an edge of equal length by translation such that inward pointing normal vectors on each edge point in opposite directions after identification. Finite-type translation surfaces are collections of finitely many polygons with only finitely many edges, whereas infinite-type translation surfaces allow for constructions containing countably many polygons allowing infinitely many edges. See, for example,~\cite{DeLaCroix-Hubert-Valdez}.

In this article, we study a particular infinite-type translation surface which we call an \emph{armadillo tail surface} or \emph{armadillo tail}. For computational convenience, we use the following polygonal representation. Place a square, denoted by $\square_1$ (``square one''), so that the lower left vertex lies at the origin and all edges are parallel to the axes. For $k \geq 1,$ glue the left side of $\square_{k+1}$ to the right side of $\square_k$ so that the bottom edge of all squares lie on the $x$-axis. We denote the side length of $\square_k$ by $l_k,$ and assume that $(l_k)$ is a strictly decreasing sequence. We then identify horizontal (vertical, resp.) edges via vertical (horizontal, resp.) translation. Bowman~\cite{bowman} and Degli Esposti--Del Magno--Lenci~\cite{ddl1},~\cite{ddl2} have built similar infinite-type translation surfaces but allowed for rectangles in lieu of squares; the surface in the former article is known as a ``stack of boxes'' and the ones in the latter, ``infinite step billiards'' and ``Italian billiards.''

The following are examples of finite-area armadillo tails.
\begin{ex} \begin{enumerate}
\item The armadillo tail where $l_k = r^{k-1},$ for $0<r<1,$ which we call a \emph{geometric armadillo tail} with parameter $r.$ It is bounded and has finite area, $\frac{1}{1-r^2}.$
    
\item The \emph{harmonic armadillo tail } where $l_k = \frac{1}{k}$ (Figure~\ref{fig:harmonic}). While the polygonal representation is not bounded in the horizontal direction, its area is $\zeta(2) = \frac{\pi^2}{6},$ finite. 
\end{enumerate}

\begin{figure}[htbp]
	\centering
	\includegraphics[width=0.7\linewidth]{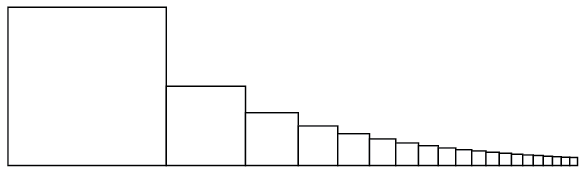}
	\caption{The harmonic armadillo tail}
 \label{fig:harmonic}
 \end{figure}
    
\end{ex}

With horizontal (vertical, resp.) edges being identified via vertical (horizontal, resp.) translation, the resulting translation surface is an infinite genus surface with one wild singularity. For background on wild singularities, see~\cite{bowman},\cite{DeLaCroix-Hubert-Valdez},~\cite{randecker2016}. The wild singularity appears infinitely many times in the polygonal representation: each vertex of the infinite-sided polygon is the same point. To make the armadillo tail a surface, we must remove the wild singularity.  Throughout the paper, we will use \emph{armadillo tail surface} to mean the construction without the wild singularity. However, when we are not using surface theory, we will often take the metric completion of the armadillo tail surface and refer to this as the \emph{armadillo tail}. Note that we use the metric that agrees with the Euclidean metric in the polygonal representation.  

In what follows, our focus is on periodic directions of the linear flow. These are directions admitting a cylinder decomposition which is, informally, a decomposition of the surface into cylinders whose interiors are foliated by closed trajectories in that direction. For infinite-type translation surfaces, there is no consensus on how to define a cylinder. See Remark \ref{rem:cyl} below. We opt for the following definition.  A \emph{cylinder} is a closed subspace of the surface whose interior is foliated by homotopic closed straight-line trajectories and whose boundary consists of \emph{saddle connections}, where a saddle connection is a line segment whose endpoints coincide with a singularity in the metric completion of the surface, but contain no singularities in the interior of the line segment. Note that saddle connections on the armadillo tail \emph{surface} will be open line segments since the only singularity is a wild singularity. A \emph{cylinder decomposition} $\mathcal{C}$ is a collection of cylinders with the following properties: 
\begin{enumerate}
\item any two distinct cylinders are either disjoint or intersect only along their boundaries, and 
\item the union of cylinders is dense on the surface, i.e. the closure of the union covers the surface.
\end{enumerate}

If the closure of the union of cylinders contains points that are not in any cylinder, those points comprise the \emph{spine} of the cylinder decomposition. If a spine is made of a single orbit of the linear flow in the direction of the cylinders (in forwards and backwards time), we call it a \emph{rigid spine}. For example, the spine may consist of a single saddle connection. If a spine is comprised of multiple orbits (possibly infinitely many), we call it a \emph{flexible spine}. For example, a spine may be comprised of a collection of multiple (possibly infinite) saddle connections. If a cylinder decomposition has an empty spine, we say it is a \emph{spineless cylinder decomposition}. Observe that the notion of a cylinder decomposition can be extended to the armadillo tail (with the singularity included) by allowing for the singularity to appear in the saddle connections defined above. We will use the term saddle connection on an armadillo tail to mean a closed line segment starting and ending at the wild singularity. 

A closed straight-line trajectory in the interior of a cylinder is called a \emph{waist curve}. The \emph{circumference} of a cylinder is the length of a waist curve, and the \emph{width} of a cylinder is the distance between the bounding saddle connections. 

\begin{remark}\label{rem:cyl} Cylinders can have higher genus on the metric completion of the surface. Because of this, one may wish to define \emph{open cylinders} where we do not include the boundary saddle connections. However, our definition allows us to distinguish saddle connections in a spine from the saddle connections in the boundaries of the cylinders in the cylinder decomposition. 
\end{remark}

Finding cylinder decompositions can be challenging, even on finite-type translation surfaces. However, given our choice of polygonal representation of the armadillo tail surface there is a simple example. Consider the cylinder decomposition $\mathcal{C}$ of an armadillo tail surface in the horizontal direction. See Figure \ref{fig:horizontal}. 

\begin{figure}[h]
	\centering
	\includegraphics[width=0.6\linewidth]{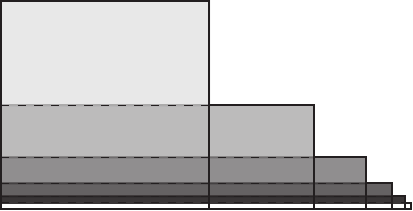}
	\caption{A cylinder decomposition of an armadillo tail in the horizontal direction}
 \label{fig:horizontal}
\end{figure}

From top to bottom, each cylinder in the cylinder decomposition becomes longer and thinner, eventually limiting to a concatenation of infinitely many saddle connections at the base of the polygonal representation. This may appear to be a flexible spine, but each of these saddle connections is a boundary component of some cylinder, namely the ``top" of a cylinder. Indeed, the bottom saddle connections are identified to the saddle connections appearing at the top of each square. This is a spineless cylinder decomposition. 

There is another spineless cylinder decomposition: the squares themselves are cylinders and the closed trajectories foliated the squares are vertical trajectories in the polygonal representation. See Figure \ref{fig:vertical}.

\begin{figure}[h]
	\centering
	\includegraphics[width=0.6\linewidth]{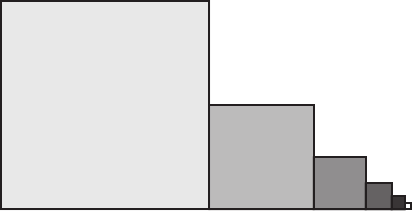}
	\caption{A cylinder decomposition of an armadillo tail in the vertical direction}
 \label{fig:vertical}
\end{figure}

\subsection{Main Results}

\begin{thm}\label{thm:one} There exists a cylinder decomposition with a rigid spine on any geometric armadillo tail surface of parameter $\frac{1}{q}$, $q \in \N\setminus \{1\}$.  
\end{thm} 

For the case $q = 2,$ see Figure \ref{fig:cyl-for-thm} for the first few cylinders and the spine.

\begin{figure}[htbp]
	\begin{center}
	\includegraphics[width=0.8\linewidth]{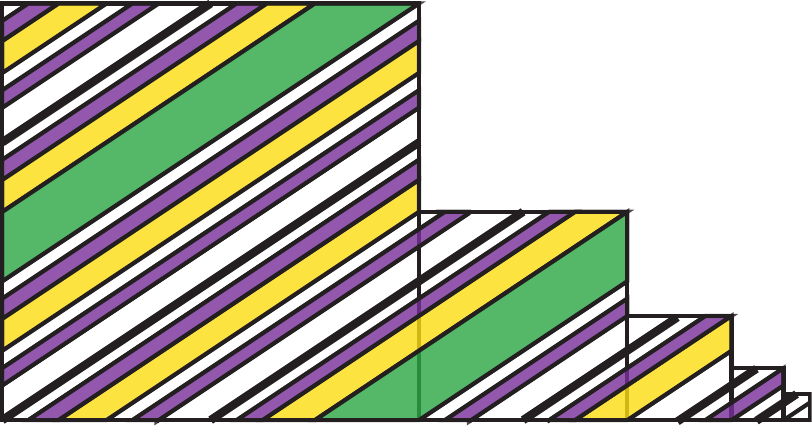}
	\caption{Cylinder decomposition on the geometric armadillo tail ($r=1/2$). The black line is the spine.}
 \label{fig:cyl-for-thm}
 \end{center}
\end{figure}

The saddle connection in the spine is identified in Section~\ref{sec:one-dir} and the cylinder decomposition construction is completed in Section~\ref{sec:cyl}. Theorem~\ref{thm:main} proves Theorem~\ref{thm:one}.

The existence of a spine is interesting in its own right. For instance, the set of rigid spines on a surface places restrictions on the Veech group. Indeed,

\begin{thm}\label{thm:discrete-veech-group} assume that a surface has a set of rigid spines such that each rigid spine is comprised of a single saddle connection. Then, if the set of saddle connections corresponding to the rigid spines developed into the plane is discrete, and at least two are linearly independent, the Veech group is discrete.
\end{thm}

This follows from the fact that a rigid spine comprised of a single saddle connection must be mapped to a rigid spine comprised of a single saddle connection under any affine diffeomorphism of the surface. Then, the same argument that shows that a discrete set of holonomy vectors implies that the Veech group of a translation surface is discrete will work here. For this argument, see for example,~\cite{HS-intro}. 

For a large class of armadillo tails, the Veech group is isomorphic to $\mathbb{Z}.$ We include a proof of this in Appendix \ref{app:veech-group}. This proof was first suggested to us by Anja Randecker. However, it is unclear if the orbit of the rigid spine under the affine group is the full set of rigid spines. This appears to be the case. Moreover, while not proven here, we have computational evidence that the periodic directions on a geometric armadillo tail of parameter $\frac{1}{q}$ can have spines with arbitrarily many saddle connections in them. Indeed, Figures \ref{fig:two_spines} and \ref{fig:two_spines_w_bsc3} show a partial cylinder decomposition with a spine comprised of two saddle connections. 

\begin{figure}[htbp] %  figure placement: here, top, bottom, or page
\centering
%\begin{minipage}{.5\textwidth}
	\centering
	\includegraphics[width=0.8\textwidth]{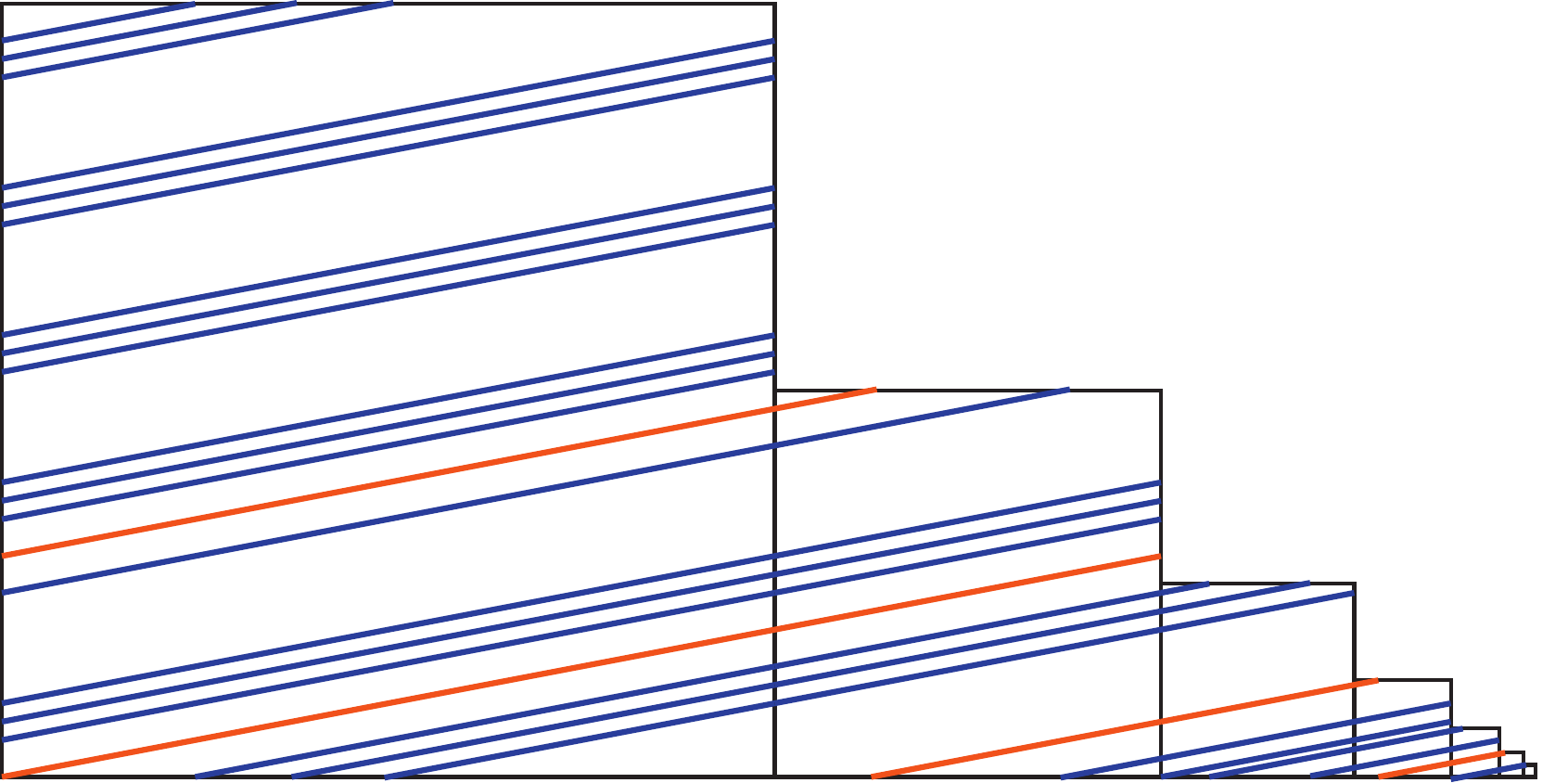}
%	\end{minipage}%
%\begin{minipage}{.5\textwidth}
%	\centering
%	\includegraphics[width=0.95\textwidth]{two_spines_w_cyl1_gray.pdf}
%	\end{minipage}
	\caption{Two saddle connections on the geometric armadillo tail with $r=1/2$.}
	\label{fig:two_spines}
\end{figure}

\begin{figure}[htbp]
    \centering
    \includegraphics[width=0.8\linewidth]{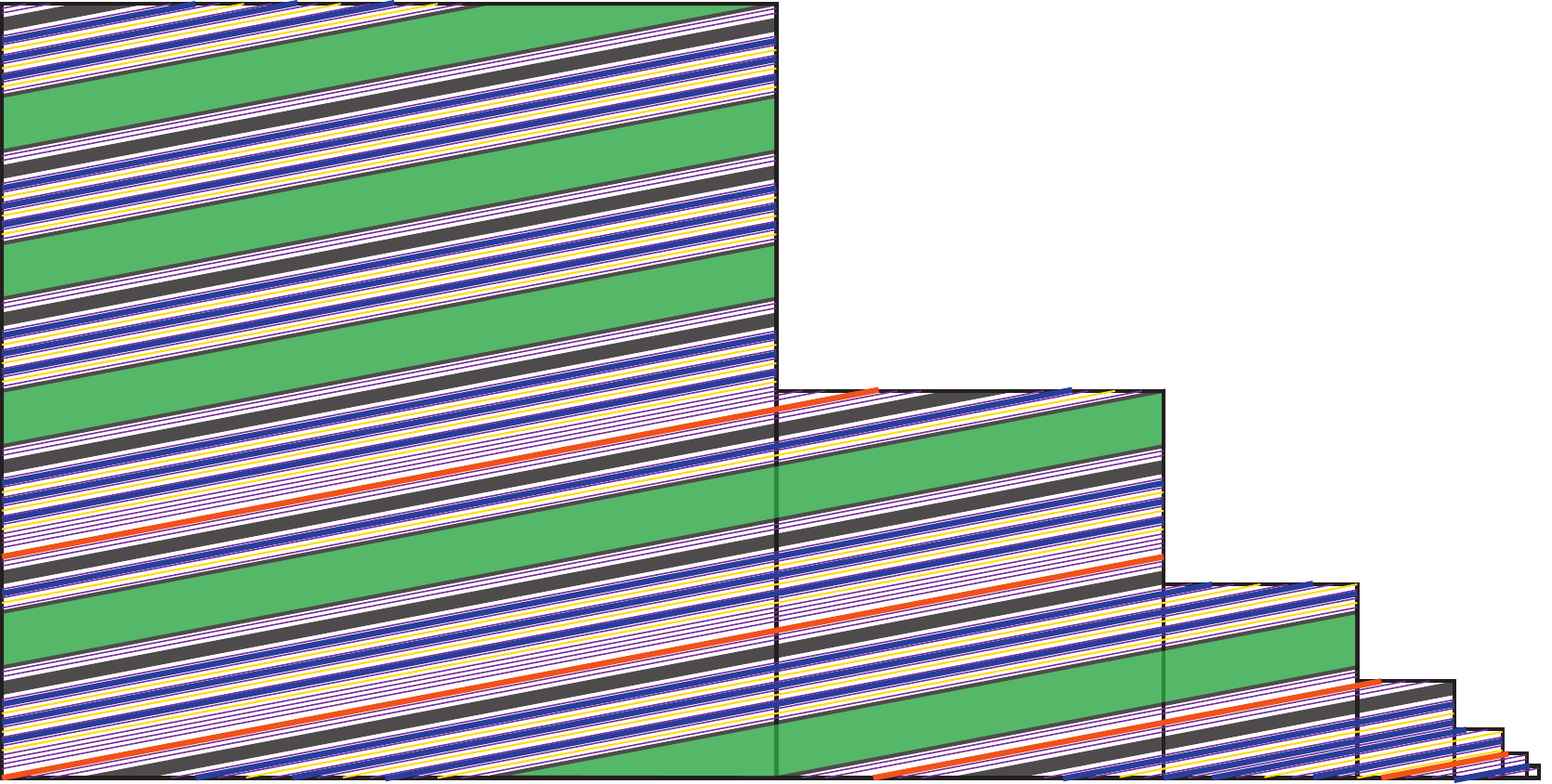}
    \caption{Cylinder decomposition on the geometric armadillo tail ($r=1/2$). The two saddle connections in the spine are marked in orange and blue. The first four cylinders are marked green, gray, yellow, purple, in the order of decreasing width.}
    \label{fig:two_spines_w_bsc3}
\end{figure}

From this perspective, one can distinguish between the various periodic directions that arise on an infinite-type translation surface by looking at the spine. The cylinder decompositions in Figure~\ref{fig:horizontal}, \ref{fig:vertical}, \ref{fig:cyl-for-thm}, and \ref{fig:two_spines_w_bsc3} are all distinct in that there is no affine diffeomorphism of the surface that maps one to the other. 

\subsection{Outline of paper} The proof of Theorem~\ref{thm:one} is constructive. We leverage the structure of the surface and interval exchange transformations to inductively construct cylinders. In Section~\ref{sec:one-dir}, we identify a curve in a special direction, which is the saddle connection in a rigid spine of a cylinder decomposition. This curve wraps around every square torus in the surface. Next, we construct a cylinder which turns out to be the widest cylinder in the cylinder decomposition. In Section~\ref{sec:bsc}, we inductively construct a collection of saddle connections which turn out to frame all of the cylinders in our cylinder decomposition. We use two techniques to do this. One is by direct calculation. The second uses an infinite interval exchange transformation (IET). See~\cite{Viana} for a description of IETs. In Section~\ref{sec:cyl}, we construct the cylinder decomposition by defining a (discontinuous) map which pushes a cylinder to a subset of the next widest cylinder in the decomposition. We ``fill in" the missing segments of the cylinder using a circle rotation argument. The key observation that makes this argument work concerns the broken symmetry of the surface ($\square_i$ is a scaled up copy of $\square_{i+1}$ for any $i \in \N,$ but there is no square that is a scaled up version of $\square_1$). In Appendix~\ref{app:veech-group}, we provide a proof that the Veech group of many armadillo tails is $\mathbb{Z}.$ In Appendix~\ref{app:area-of-cyls}, we compute the modulus and area of each cylinder in the cylinder decomposition. 

\subsection{Related work}
\subsubsection{Linear flow} 

On finite translation surfaces, the linear flow is well-studied. On every finite translation surface, there exists infinitely many closed linear trajectories \cite{masur-closed-trajectories}. Moreover, the linear flow satisfies the following property due to Masur: in each direction, the surface decomposes into subsurfaces such that each subsurface is either a cylinder or a region of the surface with a dense linear flow. Not much is known about the linear flow on infinite translation surfaces. Indeed, it is not known whether or not a closed trajectory even exists on an infinite translation surface (aside from simple counterexamples). 

With regards to the dynamics of the linear flow, it is an interesting question as to what ergodic measures are supported on an armadillo tail (geometric or otherwise). A generalization of a Veech dichotomy of this flavor was done for infinite staircase surfaces by Hooper, Hubert, and Weiss~\cite{hooper-hubert-weiss}, and recently recovered by Artigiani, Castorrini, Ravotti, and Tumarkin in ongoing work. 

\subsubsection{Affine diffeomorphisms} Bowman studies the ``geometric limit" of finite translation surfaces converging to an infinite-type translation surface~\cite{bowman-arnouxyoccoz}. Along these lines, we can think of an armadillo tail as the limit of finite-type translation surfaces. Indeed, consider the finite translation surface $X_n = \bigcup\limits_{k=1}^n \square_k$ where we make the same identifications as before, but now we identify the right edge of $\square_n$ with the bottom segment of the left edge of $\square_1.$ On each $X_n,$ there is a cylinder decomposition in the same direction that we find the rigid spine. Moreover, for all $X_n,$ there is a parabolic affine diffeomorphism that preserves this cylinder decomposition. (See \cite{hooper-hubert-weiss} for the definition of an affine diffeomorphism.) However, in the limit, the sequence of parabolic affine diffeomorphisms does not converge to any affine map on the armadillo tail surface.

%Moreover, Bowman proves a statement similar to our Theorem~\ref{thm:discrete-veech-group} (see~\cite{bowman}, Theorem 2, and in particular, Lemma 2.4). However, the set of rigid spines of a surface is a finer invariant than what Bowman uses.

\subsection{Acknowledgements}
The authors would like to thank David Aulicino, Matt Bainbridge, and Pat Hooper for helpful conversations. The authors would also like to thank Rodrigo Trevi\~no for helpful discussions, especially with regards to the existence, or lack thereof, of a cylinder decomposition in directions orthogonal to a spine direction. We are also indebted to Anja Randecker, who first suggested to us the proof of the Veech group of an armadillo tail being $\mathbb{Z}.$

\subsection{AI Declaration} No AI/LLM was used in any part of this work.

%%%%%%%%%%%%%%%%%%%%%%%%%%%%%%%%%%%%%
%%%%%%%%%%%%%%%%%%%%%%%%%%%%%%%%%%%%%
%%%%%%%%%%%%%%%%%%%%%%%%%%%%%%%%%%%%%
%%%%%%%%%%%%%%%%%%%%%%%%%%%%%%%%%%%%%
%%%%%%%%%%%%%%%%%%%%%%%%%%%%%%%%%%%%%
%%%%%%%%%%%%%%%%%%%%%%%%%%%%%%%%%%%%%
%%%%%%%%%%%%%%%%%%%%%%%%%%%%%%%%%%%%%
%%%%%%%%%%%%%%%%%%%%%%%%%%%%%%%%%%%%%

\medskip
\section{The spine and the first cylinder (or a particular direction on geometric armadillo tails)}\label{sec:one-dir}

In this section we identify a closed saddle connection on each geometric armadillo tail of parameter $r \in (0,1).$ For a subset of these parameters, this saddle connection is identified as a rigid spine of a cylinder decomposition in Section \ref{sec:cyl}. 

The following theorem requires no assumption on the parameter $r$. 

\begin{thm}\label{thm:spine} On a geometric armadillo tail with $r \in (0,1),$ there exists a closed saddle connection in the $\frac{1}{2-r}$-direction that intersects every square torus. 
\end{thm}

By $\frac{1}{2-r}$-direction, we mean the direction with slope $\frac{1}{2-r}$ relative to our polygonal representation. For example, see Figure~\ref{fig:geometric-gs}.

\begin{figure}[htbp]
	\begin{center}
	\includegraphics[width=0.7\linewidth]{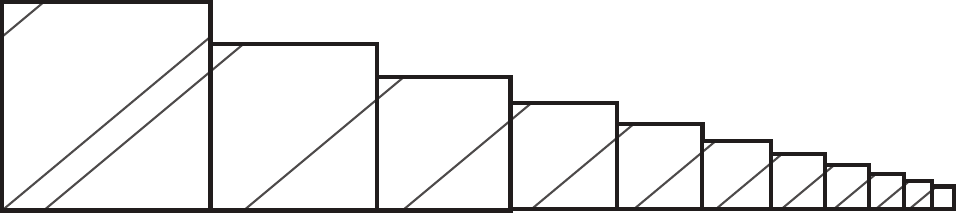}
	
	\caption{A geometric armadillo tail ($r=4/5$) with a trajectory of slope $5/6$}
    \label{fig:geometric-gs}
 \end{center}
\end{figure}

\begin{proof} We refer to the top (horizontal) edge of a square by the \emph{roof} and the right (vertical) edge of a square that is identified to a segment on the $y$-axis by the \emph{portal.}

We start from the origin, the lower left vertex of $\square_1$. Since $r<\frac{1}{2-r}<1,$ the straight line of slope $\frac{1}{2-r}$ through the origin hits the portal of $\square_1$ at point $(1,\frac{1}{2-r}).$ By identification with the $y$-axis (the left edge of $\square_1$), the trajectory continues and hits the roof of $\square_1$ at $\left(1-r, 1\right).$ By identification with the bottom edge of $\square_1,$ the trajectory continues from $\left(1-r,0\right)$ and hits the roof of $\square_2$ at $\left(1+r(1-r),r\right).$ The trajectory partitions both roofs with a fixed ratio. Due to similarity in consecutive squares, it continues to partition subsequent roofs with the same ratio. Needless to say, the trajectory wraps around every square torus without hitting any vertex. \end{proof}

In the following theorem, we show that for geometric armadillo tails with $r=\frac{1}{q}$ for $q \in \N \setminus \{1\},$ there exists not only a saddle connection but a cylinder in the $\frac{1}{2-r}$-direction. We call this cylinder $\mathrm{cyl}_1$. See Figure~\ref{fig:cyl1-on-q=2and3}.  The existence of this cylinder will be part of the base case for the induction in Section \ref{sec:cyl}. 

\begin{figure}[h]
    \centering
    \includegraphics[width=0.9\linewidth]{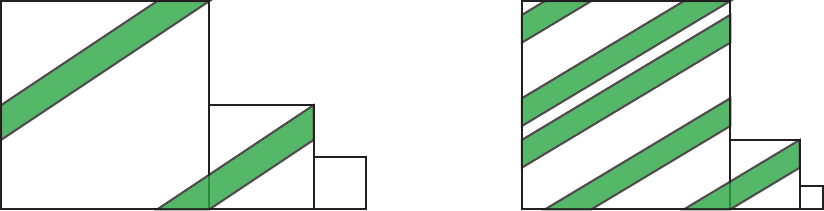}
    \caption{The cylinder $\mathrm{cyl}_1$ on a geometric armadillo tail with $r=1/2$
(left) and $1/3$ (right)}
    \label{fig:cyl1-on-q=2and3}
\end{figure}

\begin{thm}\label{thm:cyl1} Given a geometric armadillo tail with parameter $r=\frac{1}{q},$ $q\in\N\setminus\{1\},$ there exists a cylinder in the $\frac{1}{2-r}$-direction which lies entirely in $\square_1 \cup \square_2.$ \end{thm} 

There are three parts to this proof. First, we show that there is a saddle connection (in the $\frac{1}{2-r}$-direction) that lies entirely in $\square_1 \cup \square_2.$ Secondly, we will show that there is another saddle connection in $\square_1 \cup \square_2$ parallel to the first one and of the same length. In the third step, we will show that there is no saddle connection between the two saddle connections, and consequently that the interstitial space is foliated by closed geodesics, hence yielding a cylinder.

\begin{proof} 
Note, for $q = \N \setminus \{1\},$ the slope of the trajectory is $\frac{1}{2-r} = \frac{q}{2q-1}.$ 

\textbf{Step 1.} We show that the trajectory from $(1,0)$ with slope $\frac{q}{2q-1}$ stays entirely in $\square_1 \cup \square_2.$

Start from $(1,0)$ we hit the portal of $\square_2$ at $\left(1+\frac{1}{q}, \frac{1}{2q-1}\right),$ hence the first time the trajectory hits the $y-$axis is at $\left(0,\frac{1}{2q-1}\right).$ We continue and hit $\left(1, \frac{1+q}{2q-1}\right),$ which lies in the portal of $\square_1$ since $\frac{1}{q} < \frac{1+q}{2q-1} < 1.$ It is then identified to $\left(0,\frac{1+q}{2q-1}\right).$ Note that if we keep hitting the portal, the $n^{th}$ time the trajectory hits the vertical axis is at $\left(0,\frac{1}{2q-1} + \frac{(n-1) q}{2q-1}-\lfloor\frac{1}{2q-1} + \frac{(n-1) q}{2q-1}\rfloor\right).$ In fact, since $\frac{1}{2q-1} < \frac{1}{q} < \frac{2}{2q-1},$ we always hit the portal unless $\frac{1}{2q-1} + \frac{(n-1) q}{2q-1}-\lfloor\frac{1}{2q-1} + \frac{(n-1) q}{2q-1}\rfloor = \frac{1}{2q-1}.$ Pick $n=2q-2,$ then the trajectory hits the singularity at $(1,1).$ In other words, the trajectory goes through $\square_2$ exactly once at the beginning and stays in $\square_1$ until it hits $(1,1)$. The saddle connection that we constructed will serve as the ``bottom'' boundary saddle connection of $\mathrm{cyl}_1.$ 

Furthermore, this is the first time the trajectory hits a singularity. This follows from the fact that $\gcd(q,2q-1)=1.$ The trajectory hits the $y$-axis at points $$\left\{(0,y): y = \frac{(n-1)q+1}{2q-1}-\lfloor\frac{(n-1)q+1}{2q-1}\rfloor, n = 1, \ldots, 2q-3\right\}.$$ The $y$-coordinate of the elements can be rearranged as $$\left\{\frac{1}{2q-1}, \frac{1+q}{2q-1}, \frac{1+2q}{2q-1}, \ldots, \frac{1+(2q-3)q}{2q-1} \equiv 1\right\}.$$ Note that if we continue to add $\frac{q}{2q-1},$ we hit $\frac{1+(2q-2)q}{2q-1} \equiv \frac{q}{2q-1}$ and $\frac{1+(2q-1)q}{2q-1} \equiv \frac{1}{2q-1},$ which brings us back to the beginning of the sequence. In other words, the trajectory wraps around $\square_1$ exactly $2q-3$ times hitting the $y$-axis at $\left\{\left(0,\frac{i}{2q-1}\right)\right\}_{i=1,\neq q}^{2q-2}.$ 

Alternatively, notice that once the trajectory enters $\square_1$, we can encode the points that hit the portal (which are identified to the left side of $\square_1$) via a circle rotation that arises as a section of the linear flow with slope $\frac{q}{2q-1}$. We use this perspective in Steps 2 and 3 below. 

\textbf{Step 2.} Next, we construct a saddle connection which will be the ``top" boundary saddle connection of a cylinder. To do this, we will take the bottom saddle connection and show that if we shift the saddle connection vertically (in the polygonal representation) by $\frac{q-1}{q(2q-1)}$, that we find another saddle connection. A consequence of Step 3, which follows, is that $\frac{q-1}{q(2q-1)}$ is the \emph{skew-width} of this cylinder, the vertical distance (with respect to the polygonal representation) between the two saddle connections.

Take the set of points where the bottom saddle connection hits the $y$-axis and add $\frac{q-1}{q(2q-1)}:$ $$\left\{\dfrac{i}{2q-1}\right\}_{i=1,\neq q}^{2q-2} + \dfrac{q-1}{q(2q-1)} = \left\{\frac{(i+1)q-1}{q(2q-1)}\right\}_{i=1,\neq q}^{2q-2},$$ where a set $+$ number denotes adding the number to each element of the set. We split this set into three cases: (1) $i=1, \ldots , q-2,$ (2) $i=q-1,$ and (3) $i=q+1, \ldots , 2q-2.$ 

Let $T:[0,1]/_\sim \to [0,1]/_\sim$ be a circle rotation where $T(x) = x + \frac{q}{2q-1}.$ Note that the circle rotation is a section of the linear flow on $\square_1$, provided we never enter $\square_2$. We can guarantee that the linear flow does not enter $\square_2$ provided the iterates of the circle rotation are greater than or equal to $\frac{1}{q}$. In fact, every element in the set $\left\{\frac{(i+1)q-1}{q(2q-1)}\right\}_{i=1,\neq q}^{2q-2}$ is greater than or equal to $\frac{1}{q}$. Moreover, the image of these points under $T$ is always greater than $\frac{1}{q}$, except for the image of the point corresponding to case (2): $i=q-1$. We will address this when it arises.

First, observe that $T$ maps points corresponding to case (1) to points corresponding to case (3). Indeed, take points in case (1), ($i=1, \ldots, q-2$), and apply $T$. We have $$T\left(\dfrac{(i+1)q-1}{q(2q-1)}\right) = \dfrac{(i+1+q)q-1}{q(2q-1)} \in \, \left\{\dfrac{(j+1)q-1}{q(2q-1)}\right\}_{j=q+1}^{2q-1}.$$

Next, observe that the points in case (3) map to points in case (1), with one exception, in which case the image of point corresponds to the point in case (2). Take points in case (3), ($i=q+1, \ldots, 2q-2$) and apply $T$. For all cases except $i=2q-2$ we have $$T\left(\dfrac{(i+1)q-1}{q(2q-1)}\right) = \dfrac{(i+1+q)q-1}{q(2q-1)} \equiv \dfrac{(i-q+2)q-1}{q(2q-1)} \in \, \left\{\dfrac{(j+1)q-1}{q(2q-1)}\right\}_{j=1}^{q-2}.$$ If $i=2q-2,$ then we have $T\left(\frac{(2q-1)q-1}{q(2q-1)}\right)= \frac{q^2-1}{q(2q-1)} \in \, \left\{\frac{(i+1)q-1}{q(2q-1)}\right\}_{i=q-1},$ which corresponds to the point in case (2). 

Lastly, consider the point in case (2), where $i=q-1.$ We have $$T\left(\dfrac{q^2-1}{q(2q-1)}\right) = \dfrac{2q^2-1}{q(2q-1)} \equiv \dfrac{q-1}{q(2q-1)} < \dfrac{1}{q}.$$ Since the trajectory hits the right side of $\square_1$ below the portal, we continue into $\square_2$ and hit $\left(1+\frac{1}{q}, \frac{1}{q}\right)$, the singularity at the upper right vertex of $\square_2$. 

Now, observe that if we begin with $i=1$, the sequence of iterates of $T$ will include every point in case (1) and case (3), and then hit the point in case (2). In other words, the suspension of the linear flow is a closed saddle connection that passes through $\left(0, \frac{1}{q}\right)$ that lies entirely in $\square_1$ and $\square_2.$ 

\textbf{Step 3.} Lastly, we show that between these two saddle connections there is no other saddle connection in with slope $\frac{q}{2q-1}$. In other words, there are no saddle connections that hit the vertical axis between $\left(0,\frac{1}{2q-1}\right)$ and $\left(0,\frac{1}{q}\right).$ 

Define the intervals along the $y$-axis with $y$-coordinates in $(\frac{i}{2q-1}, \frac{i}{2q-1} + \frac{q-1}{q(2q-1)})$, for $i \in \{1, 2, \ldots, 2q-2\} \setminus\{q\}$. The lower bound in each interval coincides with an intersection of the (bottom) saddle connection constructed in step 1 with the $y$-axis. Similarly, the upper bound coincides with an intersection of the (top) saddle connection constructed in Step 2 with the $y$-axis. 

First, observe that the intervals do not contain a singularity. The only appearances of the wild singularity along the $y$-axis for $y>\frac{1}{2q-1}$ occur for $y = \frac{1}{q}$ and $y=1$, neither of which land inside any of the intervals. 

Now, fix any $0 < \varepsilon < \frac{q-1}{q(2q-1)}$ and consider the collection of points $(0,y_i)$ for $y_i  = \frac{i}{2q-1} + \varepsilon$, for $i \in \{1, 2, \ldots, 2q-2\} \setminus\{q\}.$ (There is one point in each of the intervals.) By applying the map $T$ to the point in the interval corresponding to $i=1$, we see that the image is the point in the collection corresponding to the  interval $i = q+1$. The argument is the same as the one given in Step (2). We continue applying the map $T$ until we reach the last interval in the set, $(0,y_{q-1})$. Here the image of the circle rotation contains the singular point infinitely many times, but this is because the circle rotation no longer applies. Indeed, $T(y_{q-1}) < \frac{1}{q}$ which means that the suspended flow is actually entering $\square_2$. (This is identical to the situation described in case (2) in step 2.) If we flow in the linear direction from $(0,y_{q-1})$, we hit the portal in $\square_2$ at $(1+\frac{1}{q}, y_1)$, which means that the linear flow starting at $(0,y_1)$ is a closed geodesic. 

Moreover, since the choice of $\varepsilon$ allows for any point in the interval, the suspension of these intervals consists of closed geodesics.

In conclusion, on a geometric armadillo tail with parameter $r=\frac{1}{q},$ there exists a cylinder that lies entirely in $\square_1$ and $\square_2.$ We call this cylinder $\mathrm{cyl}_1.$
\end{proof}

\begin{remark} The assumption that $r=\frac1q$ cannot be removed. For example, if $r=2/3,$ the saddle connection starting from $(1,0)$ in the $\frac{1}{2-r}$-direction hits the wild singularity at $(1+r+r^2,r^2),$ the upper-right vertex of $\square_3,$ and is contained in $\square_1\cup\square_2\cup\square_3.$\end{remark}

%%%%%%%%%%%%%%%%%%%%%%%%%%%%%%%%%%%%%
%%%%%%%%%%%%%%%%%%%%%%%%%%%%%%%%%%%%%
%%%%%%%%%%%%%%%%%%%%%%%%%%%%%%%%%%%%%
%%%%%%%%%%%%%%%%%%%%%%%%%%%%%%%%%%%%%
%%%%%%%%%%%%%%%%%%%%%%%%%%%%%%%%%%%%%
%%%%%%%%%%%%%%%%%%%%%%%%%%%%%%%%%%%%%
%%%%%%%%%%%%%%%%%%%%%%%%%%%%%%%%%%%%%
%%%%%%%%%%%%%%%%%%%%%%%%%%%%%%%%%%%%%

\medskip

\section{Framing (or finding the bottom saddle connection of each cylinder by induction)}\label{sec:bsc}

In this section, we construct concatenations of saddle connections that bound the cylinders in our cylinder decomposition. We label them $\mathrm{bsc}_{k}$, and in Section~\ref{sec:cyl}, these sets will be realized as the \textbf{b}ottom boundary \textbf{s}addle \textbf{c}onnections of the $k^{th}$ widest cylinder in the cylinder decomposition. To reduce the notational complexity, and since all of the following statements fix $r = \frac1q$, we suppress the dependence of $\mathrm{bsc}_k$ on $q$.

Let $\mathrm{bsc}_1$ be the single saddle connection constructed in Step 1 of Theorem~\ref{thm:cyl1}. We construct $\mathrm{bsc}_2$ as a concatenation of two saddle connections, one being the saddle connection constructed in Step 2 of Theorem~\ref{thm:cyl1} (the top of the $\mathrm{cyl}_1$), the other being a saddle connection that does not bound $\mathrm{cyl}_1$. This second saddle connection will be realized as the part of the bottom boundary of the second widest cylinder that is not adjacent to $\mathrm{cyl}_1$. We call this second saddle connection $\mathrm{bsc}_2^{\prime}$, and construct it in Lemma~\ref{lem:bsc2}. See Figure~\ref{fig:bsc2prime}. In fact, $\mathrm{bsc}_k$ will always be a concatenation of two saddle connections for $k \geq 2.$ 

Assuming Lemma~\ref{lem:bsc2}, we can write $\mathrm{bsc}_{2} = \left(\mathrm{bsc}_{1} + \frac{q-1}{q(2q-1)}\right) \oplus \mathrm{bsc}^{\prime}_2$, where $\oplus$ means concatenate, and addition means shifting the saddle connections in $\mathrm{bsc}_1$ vertically by $\frac{q -1 }{q(2q-1)}$. To shift saddle connections vertically, we remove the singular point from the saddle connections, then vertically shift, and then take the closure. Indeed, a vertical shift is not well-defined at a singular point. Also, note that $\frac{q-1}{q(2q-1)}$ is the skew-width of the cylinder $\mathrm{cyl}_1$, where the skew-width is the vertical distance with respect to the polygonal representation between the saddle connections bounding the cylinder.

We will use this idea to inductively construct $\mathrm{bsc}_k = \left(\mathrm{bsc}_{k-1} + \frac{q-1}{q^{k-1}(2q-1)}\right) \oplus \mathrm{bsc}_k^{\prime}$. We note that $\frac{q-1}{q^{k-1}(2q-1)}$ will be realized as the skew-width of the $(k-1)$th cylinder. In Subsection~\ref{ssec:bsc_k^prime}, we will construct $\mathrm{bsc}_k^{\prime}$ by carefully tracking segments along the linear flow. In Subsection~\ref{ssec:bsc_k}, we will vertically shift each $\mathrm{bsc}_{k-1}^{\prime}$ by $\frac{q-1}{q^{k-1}(2q-1)}$, and extend this along the linear flow to a single saddle connection. In lieu of carefully tracking the segments along the linear flow, we will use an interval exchange transformation. 

\subsection{Constructing $\mathrm{bsc}_k^{\prime}$}\label{ssec:bsc_k^prime}

In the following lemma, for every $q \in \N\setminus \{1\}$, we construct $\mathrm{bsc}_2^{\prime}$. We use the linear flow with slope $\frac{q}{2q-1}$, but the reader may also interpret the work as identifying fixed points of an interval exchange transformation (IET) arising from a section of the linear flow. We will use this perspective in Subsection \ref{ssec:bsc_k}.

\begin{figure}
    \centering
    \includegraphics[width=1\linewidth]{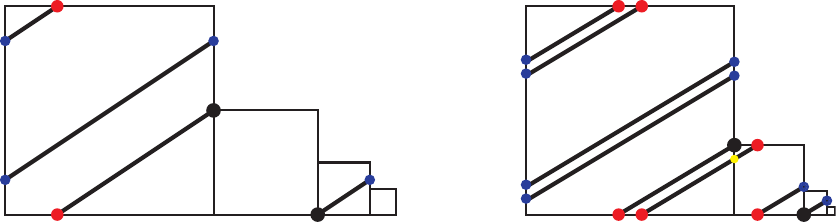}
    \caption{ $\text{bsc}_2'$ for $r=1/2$ (left) and 1/3 (right). Points that go through the portal are marked blue, points that go through the roof are marked red. The appearance of the wild singularity on the curve is marked in black. A special point referenced in a proof below is in yellow. The dots are sized in the order black, red, blue, and yellow, where black is largest.}
    \label{fig:bsc2prime}
\end{figure}

\begin{lem}\label{lem:bsc2} Given any geometric armadillo tail with parameter $r=\frac{1}{q},$ $q\in \N \setminus \{1\},$ there exists a saddle connection that we call $\mathrm{bsc}_2'$.\end{lem}

The content of Lemma~\ref{lem:bsc2} is the construction of a specific saddle connection to which we assign the name $\mathrm{bsc}_2'.$ 

\begin{proof} We use the following notations:
\begin{itemize}
    \item $\xrightarrow{x}$ indicates the linear flow where the horizontal displacement is $x,$
    \item $\extoverset[6pt]{\text{portal}(\#)}{\sim}$ indicates that the flow hit a portal of $\square_\#,$ i.e., the $x$-coordinate is $1+\cdots+\frac{1}{q^{\#-1}}$ and the $y$-coordinate lies between $\frac{1}{q^\#}$ and $\frac{1}{q^{\#-1}},$ hence it is identified to the corresponding point on the $y$-axis. For $q=2,3,$ these points are marked in blue in Figure~\ref{fig:bsc2prime}.
    \item $\extoverset[6pt]{\text{roof}(\#)}{\sim}$ indicates that the flow hits the roof of $\square_{\#}$ which is identified to the base of $\square_{\#}$, hence we adjust the $y$-coordinate. For $q=2,3,$ these points are marked in red in Figure~\ref{fig:bsc2prime}.
\end{itemize}

We start from $\left(1+\frac{1}{q},0\right)$ and flow in the $\frac{q}{2q-1}$-direction to construct $\mathrm{bsc}_2'$. 

\begin{align*}
\left(1+\frac{1}{q},0\right) & \xrightarrow{\frac{1}{q^2}} \left(1+\frac{1}{q}+\frac{1}{q^2}, \frac{1}{q (2q-1)}\right) \extoverset[6pt]{\text{portal}(3)}{\sim} \left(0, \frac{1}{q (2q-1)}\right) \\
& \xrightarrow{1} \left(1, \frac{1+q^2}{q (2q-1)}\right)  \extoverset[6pt]{\text{portal}(1)}{\sim} \left(0, \frac{1+q^2}{q (2q-1)}\right) \extoverset[6pt]{\text{roof}(1)}{\sim} \left(1, \frac{1+q}{q (2q-1)}\right).
\end{align*}

All operations above apply to all $q \in \N \setminus \{1\}.$ Observe that $\frac{1+q}{q (2q-1)} = \frac{1}{q}$ if $q=2,$ and $\frac{1+q}{q (2q-1)} < \frac{1}{q}$ if $q > 2.$ That is, if $q=2,$ the linear flow in the $\frac{q}{2q-1}$-direction from $\left(1+\frac{1}{q},0\right)$ to $\left(1, \frac{1+q}{q (2q-1)}\right)$ is a saddle connection, $\mathrm{bsc}_2'$. % See left of Figure~\ref{fig:bsc2prime}. 
If $q > 2,$ we have $\frac{1+q}{q (2q-1)} < \frac{1}{q}.$ Hence we do not yet have a saddle connection, and we enter $\square_2$. See Figure \ref{fig:bsc2prime} where this point is marked yellow in the case of $q=3$ and black in the case of $q=2.$ 

Continuing from the last expression,

\begin{align*}
\left(1, \frac{1+q}{q (2q-1)}\right) & \xrightarrow{\frac{1}{q}} \left(1+\frac{1}{q},\frac{1+2q}{q (2q-1)}\right) \extoverset[6pt]{\text{roof}(2)}{\sim} \left(1+\frac{1}{q},\frac{2}{q (2q-1)}\right) \\
&\extoverset[6pt]{\text{portal}(2)}{\sim} \left(0,\frac{2}{q (2q-1)}\right) \xrightarrow{1} \left(1,\frac{2+q^2}{q (2q-1)}\right) \extoverset[6pt]{\text{portal}(1)}{\sim} \left(0,\frac{2+q^2}{q (2q-1)}\right) \\
&\xrightarrow{1} \left(1,\frac{2+2 q^2}{q (2q-1)}\right)  \extoverset[6pt]{\text{roof}(1)}{\sim} \left(1,\frac{2+q}{q (2q-1)}\right).
\end{align*}

Note that $\frac{2+q}{q (2q-1)} = \frac{1}{q}$ if $q=3,$ and $\frac{2+q}{q (2q-1)} < \frac{1}{q}$ if $q>3.$ Hence we have a saddle connection $\mathrm{bsc}_2'$ for $q=3$. In the induction that follows, we show that the number of times the trajectory hits portal(1) is $q-1.$

Now, we use the fact that $\frac{n+q}{q (2q-1)} = \frac{1}{q}$ if $q=n+1$ for any $n \in \N$. If $q>n+1$ (we have not iterated $n$ enough) then $\frac{n+q}{q (2q-1)} < \frac{1}{q}$ and the flow always re-enters $\square_2.$ That is, the linear flow has not yielded a saddle connection, and continuing, 

\begin{align*}
\left(1, \frac{n+q}{q (2q-1)}\right) & \xrightarrow{1/q} \left(1 + \frac{1}{q}, \frac{n+2 q}{q (2q-1)}\right) \extoverset[6pt]{\text{roof}(1)}{\sim} \left(1 + \frac{1}{q}, \frac{n+1}{q (2q-1)}\right) \\
& \extoverset[6pt]{\text{portal}(2)}{\sim} \left(0, \frac{n+1}{q (2q-1)}\right)  \xrightarrow{1} \left(1, \frac{n+1+q^2}{q (2q-1)}\right) \extoverset[6pt]{\text{portal}(1)}{\sim} \left(0, \frac{n+1+q^2}{q (2q-1)}\right) \\
& \xrightarrow{1} \left(1, \frac{n+1+2q^2}{q (2q-1)}\right)  \extoverset[6pt]{\text{roof}(1)}{\sim} \left(1, \frac{n+1+q}{q (2q-1)}\right),
\end{align*}

\noindent we see that we iterate $n$, and the process continues until $n+1=q.$
\end{proof}

In the following theorem, for every $q \in \N\setminus \{1\}$, we construct $\mathrm{bsc}_k^{\prime}$ for every integer $k>2$. The previous lemma serves as the base case for an induction proof that shows the existence of this saddle connection.

\begin{thm}\label{thm:bsck'} Given a geometric armadillo tail with parameter $r=\frac{1}{q},$ $q \in \N \setminus \{1\},$ there exists a saddle connection that we call $\mathrm{bsc}_k'$ for all $k$.\end{thm}

\begin{proof} We use the same notations we did in the previous lemma.

\begin{align*}
\left(1+ \cdots + \frac{1}{q^{k-1}},0\right) & \xrightarrow{1/q^k} \left(1 + \cdots + \frac{1}{q^k},\frac{1}{q^{k-1} (2q-1)}\right) \\
&\extoverset[6pt]{\text{portal}(k+1)}{\sim} \left(0,\frac{1}{q^{k-1} (2q-1)}\right)  \xrightarrow{1} \left(1,\frac{1 + q^k}{q^{k-1} (2q-1)}\right) \\
& \xrightarrow{1} \left(1,\frac{1 + 2 q^k}{q^{k-1} (2q-1)}\right) \\
& \extoverset[6pt]{\text{roof}(1)}{\sim} \left(1,\frac{1 + q^{k-1}}{q^{k-1} (2q-1)}\right).\end{align*}

The last point above is a singularity if $q=2$ and $k=2,$ however, we assume $k > 2$ and continue.

\begin{align*}
\cdots & \xrightarrow{1/q} \left(1 + \frac{1}{q},\frac{1 + 2 q^{k-1}}{q^{k-1} (2q-1)}\right) \extoverset[6pt]{\text{roof}(2)}{\sim} \left(1 + \frac{1}{q},\frac{1 + q^{k-2}}{q^{k-1} (2q-1)}\right).\end{align*}

Again, the last expression is a singularity if $q=2$ and $k=3.$

\noindent \textbf{Induction hypothesis} For $l$ small enough, assume $$\left(1+ \cdots + \frac{1}{q^l},\frac{1 + q^{k-l-1}}{q^{k-1} (2q-1)}\right) = \left(1+ \cdots + \frac{1}{q^l}, \frac{1}{q^{l+1}}\right),$$ a singularity if $q=2.$ 

\noindent \textbf{Inductive step} We show that $$\left(1+ \cdots + \frac{1}{q^{l+1}},\frac{1 + q^{k-l-2}}{q^{k-1} (2q-1)}\right) = \left(1+ \cdots + \frac{1}{q^{l+1}},\frac{1}{q^{l+2}}\right)$$ if $q=2.$ 

For $q > 2,$ we then have

$$\begin{array}{rl}
 & \xrightarrow{1/q^{l+1}} \left(1+ \cdots + \dfrac{1}{q^{l+1}},\dfrac{1 + 2 q^{k-l-1}}{q^{k-1} (2q-1)}\right) \\
 & \extoverset[6pt]{\text{roof}(l+2)}{\sim} \left(1+ \cdots + \dfrac{1}{q^{l+1}},\dfrac{1 + q^{k-l-2}}{q^{k-1} (2q-1)}\right),\end{array}$$ and the last expression is a singularity if $q=2,$ and this proves our hypothesis.

Let $l=k-3,$ then the last expression becomes $\left(1+ \cdots + \frac{1}{q^{k-2}},\frac{1 + q}{q^{k-1} (2q-1)}\right),$ which is a singularity for $q=2.$ For $q >2,$ we have

$$\begin{array}{rl}
& \left(1+ \cdots + \dfrac{1}{q^{k-2}},\dfrac{1 + q}{q^{k-1} (2q-1)}\right) \xrightarrow{1/q^{k-1}} \left(1+ \cdots + \dfrac{1}{q^{k-1}},\dfrac{1 + 2 q}{q^{k-1} (2q-1)}\right) \\
& \extoverset[6pt]{\text{roof}(k)}{\sim} \left(1+ \cdots + \dfrac{1}{q^{k-1}},\dfrac{2}{q^{k-1} (2q-1)}\right) \extoverset[6pt]{\text{portal}(k)}{\sim} \left(0,\dfrac{2}{q^{k-1}(2q-1)}\right) \\
& \xrightarrow{1} \left(1,\dfrac{2+q^k}{q^{k-1}(2q-1)}\right) \extoverset[6pt]{\text{portal}(1)}{\sim} 
\left(1,\dfrac{2+2q^k}{q^{k-1}(2q-1)}\right) \extoverset[6pt]{\text{roof}(1)}{\sim} \left(1,\dfrac{2+q^{k-1}}{q^{k-1}(2q-1)}\right) \end{array}$$ which is a singularity if $q=3$ and $k=2.$

We claim that the rest follows from the induction technique used in Lemma~\ref{lem:bsc2}.
\end{proof}

\subsection{Constructing $\mathrm{bsc}_k$}\label{ssec:bsc_k} 

To finish the construction of $\mathrm{bsc}_k$, in lieu of tracking trajectories through portals and roofs, we will use an interval exchange transformation. If we take a section of the linear flow along the $x=0$ in the polygonal representation, we see an infinite interval exchange transformation. However, as we will see, the saddle connections in $\mathrm{bsc}_k$ never extend beyond $k+1$-squares, which means we will only need to use what is essentially a $2k$-IET.

To sample how this works, we will construct $\mathrm{bsc}_3$ using the map $T_3: [0,1)\setminus U \to [0,1) \setminus V$, where $U = \left(\dfrac{q^3 - 1}{q^2(2q-1)},\dfrac{q^4 + q - 1}{q^3(2q-1)}\right)$ and $V = \left(0,\dfrac{1}{q^3}\right)$: 

\begin{equation*}
T_3(x) =    \begin{cases}
            x + \dfrac{q}{2q-1},& x \in \left[0,\dfrac{q-1}{2q-1}\right)\\
            0, & x = \dfrac{q-1}{2q-1} \\
            x + \dfrac{2-q}{2q-1}, & x \in \left(\dfrac{q-1}{2q-1}, \dfrac{q^2-1}{q(2q-1)}\right]\\
            x + \dfrac{2-q^2}{q(2q-1)}, & x \in \left(\dfrac{q^2-1}{q(2q-1)},\dfrac{q^3-1}{q^2(2q-1)}\right]\\
            x + \dfrac{1-q^3}{q^2(2q-1)}, & x \in \left[\dfrac{q^4+q-1}{q^3(2q-1)},\dfrac{q^3+q-1}{q^2(2q-1)} \right)\\
            x + \dfrac{1-q^2}{q(2q-1)}, & x \in \left[\dfrac{q^3+q-1}{q^2(2q-1)},\dfrac{q^2+q-1}{q(2q-1)} \right)\\
            x + \dfrac{1-q}{2q-1}, & x \in \left[\dfrac{q^2+q-1}{q(2q-1)},1 \right).\\
            \end{cases}
\end{equation*}

Both $\mathrm{bsc}_3$ and $T_3$ depend on the parameter $q$. In Figure~\ref{fig:colored iet}, we give a visual description of the map $T_3$. Note that the black region corresponds to where $T_3$ is left undefined.

\begin{figure}[htbp]
    \centering
    \includegraphics[width=0.7\linewidth]{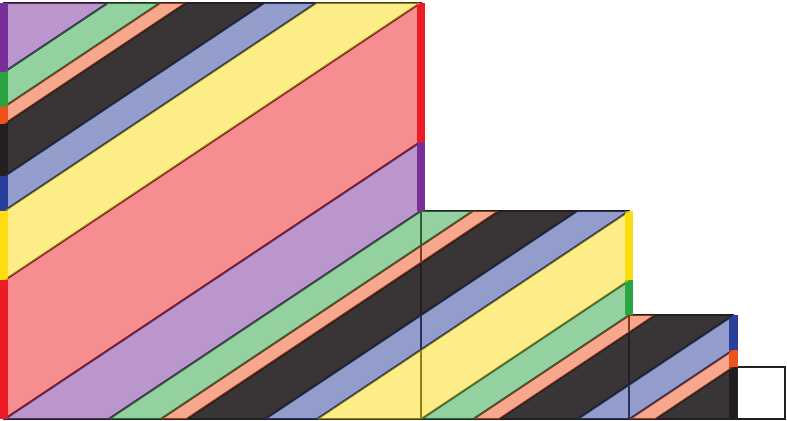}
    \caption{$6$-IET on $X_3$}
    \label{fig:colored iet}
\end{figure}

\begin{lem}\label{lem:bsc-base-case} Let $\mathrm{bsc}_{3} = \left(\mathrm{bsc}_2 + \frac{q-1}{q^{2}(2q-1)}\right) \oplus \mathrm{bsc}^{\prime}_{3}$. Then $\mathrm{bsc}_{3}$ is a concatenation of two saddle connections. 
\end{lem}

\begin{proof} Clearly, $\mathrm{bsc}^{\prime}_3$ is one of the two saddle connections. We show that $\mathrm{bsc}_2 + \frac{q-1}{q^{2}(2q-1)}$ yields a single saddle connection. 

Observe that $\mathrm{bsc}^{\prime}_2$ is one of the two saddle connections in $\mathrm{bsc}_2$. We can shift $\mathrm{bsc}^{\prime}_2$ vertically by $\frac{q-1}{q^2(2q-1)}$. When we take the closure of the vertical shift, we see endpoints corresponding to $\left(1+\frac{1}{q},\frac{q-1}{q^2(2q-1)}\right)$ and $\left(1,\frac{1}{q} + \frac{q-1}{q^2(2q-1)}\right) = \left(1,\frac{2q^2-1}{q^2(2q-1)}\right)$, neither of which are singular points. (See the construction of $\mathrm{bsc}_2^{\prime}$ in Lemma \ref{lem:bsc2}.) The segment does not cross over the singular point as we continuously push the segment vertically, however, we see that the segment intersects that singular point at exactly one point: $\left(1 + \frac{1}{q} + \frac{1}{q^2}, \frac{1}{q(2q-1)} + \frac{q-1}{q^2(2q-1)}\right) = \left(1 + \frac{1}{q} + \frac{1}{q^2}, \frac{1}{q^2}\right)$. 

Further note that $\left(1,\frac{2q^2-1}{q^2(2q-1)}\right)$ is identified to a point on the section $x=0$ since $\left(1,\frac{2q^2-1}{q^2(2q-1)}\right) \sim \left(0, \frac{2q^2-1}{q^2(2q-1)}\right)$. Similarly, the singular point $\left(1 + \frac{1}{q} + \frac{1}{q^2}, \frac{1}{q^2}\right) \sim \left(0,\frac{1}{q^2}\right)$. We will show that the linear flow connects $\left(0, \frac{2q^2-1}{q^2(2q-1)}\right)$ to $(0,\frac{1}{q^2})$ using $T_3$, and in doing so, we will see that all iterates of the map avoid the singular point. 

Observe that when $q = 2$, $T_3\left(\frac{2q^2-1}{q^2(2q-1)}\right) = \frac{1}{q^2}$, and we are done. For general $q$, we make the following observation. 

\begin{cl} For any $q \in \N \setminus\{0\}$, any $2 \leq n \leq q$, $$ T_3^{2n-4}\left(\dfrac{2q^2-1}{q^2(2q-1)}\right) = \dfrac{nq^2 - 1}{q^2(2q-1)}.$$
\end{cl}

\noindent The proof of the claim follows from induction, where $q=2$ is our base case. Indeed, 

\begin{equation*}
T_3^{2n-4}\left(\dfrac{2q^2-1}{q^2(2q-1)}\right) = T_3^2 \circ T_3^{2(n-1)-4}\left(\dfrac{2q^2-1}{q^2(2q-1)}\right) = T_3^2\left(\dfrac{(n-1)q^2-1}{q^2(2q-1)}\right)
\end{equation*}

\noindent by the inductive hypothesis. Observe that for any $n \leq q$, $$\dfrac{(n-1)q^2-1}{q^2(2q-1)} \in \left[0, \frac{q-1}{2q-1}\right)$$ since $(n-1)q^2 - 1 < q^3-q^2$ for any $n \leq q$. Consequently,  

\begin{equation*}
T_3^2\left(\dfrac{(n-1)q^2-1}{q^2(2q-1)}\right) = T_3\left( \dfrac{(n-1)q^2-1}{q^2(2q-1)} + \frac{q}{2q-1} \right) = T_3\left( \dfrac{q^3+(n-1)q^2-1}{q^2(2q-1)}\right).
\end{equation*}

\noindent Now, observe that for any $2 \leq n \leq q$, $$\dfrac{q^3+(n-1)q^2-1}{q^2(2q-1)} \in \left[\dfrac{q^2+q-1}{q(2q-1)},1 \right).$$ Indeed, $$\dfrac{q^2 + q - 1}{q(2q-1)} \leq \dfrac{q^3+(n-1)q^2-1}{q^2(2q-1)} < 1$$ since, first, $1 \leq (n-2)q^2 + q$ for any $n \geq 2$, and second, $nq^2 < 2q^3 + 1$ for $n \leq q$.

Consequently, 

\begin{equation*}
T_3\left( \dfrac{q^3+(n-1)q^2-1}{q^2(2q-1)}\right) = \dfrac{q^3+(n-1)q^2-1}{q^2(2q-1)} + \dfrac{1-q}{2q-1} = \dfrac{nq^2 - 1}{q^2(2q-1)},
\end{equation*}
\noindent as desired. This completes the proof of the claim. 

To complete the proof of the Lemma, take $n=q$ and we have $$T_3^{2q-4}\left(\dfrac{2q^2-1}{q^2(2q-1)}\right) = \dfrac{q^3 - 1}{q^2(2q-1)}$$ and since $\dfrac{q^3 - 1}{q^2(2q-1)} \in \left(\dfrac{q^2 - 1}{q(2q-1)},\dfrac{q^3 - 1}{q^2(2q-1)}\right]$, $$T_3\left(\dfrac{q^3 - 1}{q^2(2q-1)}\right) = \dfrac{q^3 - 1}{q^2(2q-1)} + \dfrac{2-q^2}{q(2q-1)} = \frac{1}{q^2},$$ which completes the proof. 
\end{proof}

An analogous proof works for $\mathrm{bsc}_k$. However, we will need to use the map $T_k: [0,1) \setminus U \to [0,1)\setminus V$, where $U = \left(\dfrac{q^k - 1}{q^{k-1}(2q-1)},\dfrac{q^{k+1} + q - 1}{q^{k}(2q-1)}\right)$ and $V = \left(0, \dfrac{1}{q^k}\right)$ defined by: 

\begin{equation*}
T_k(x) =    \begin{cases}
            x + \dfrac{q}{2q-1},& x \in \left[0,\dfrac{q-1}{2q-1}\right)\\
            0, & x = \dfrac{q-1}{2q-1} \\
            x + \dfrac{2-q}{2q-1}, & x \in \left(\dfrac{q-1}{2q-1}, \dfrac{q^2-1}{q(2q-1)}\right]\\
            x + \dfrac{2-q^2}{q(2q-1)}, & x \in \left(\dfrac{q^2-1}{q(2q-1)},\dfrac{q^3-1}{q^2(2q-1)}\right]\\
            \vdots & \\ 
            x + \dfrac{2-q^{k-i+1}}{q^{k-i}(2q-1)}, & x \in \left( \dfrac{q^{k-i+1} - 1}{q^{k-i}(2q-1)}, \dfrac{q^{k-i + 2} - 1}{q^{k-i+1}(2q-1)} \right]\\
            \vdots & \\
            x + \dfrac{2-q^{k-1}}{q^{k-2}(2q-1)}, & x \in \left( \dfrac{q^{k-1}-1}{q^{k-2}(2q-1)}, \dfrac{q^{k}-1}{q^{k-1}(2q-1)} \right] \\
            x + \dfrac{1-q^k}{q^{k-1}(2q-1)}, & x \in \left[ \dfrac{q^{k+1}+q-1}{q^k(2q-1)},  \dfrac{q^{k}+q-1}{q^{k-1}(2q-1)} \right) \\
            \vdots & \\
            x + \dfrac{1-q^{k-j+1}}{q^{k-j}(2q-1)}, & x \in \left[ \frac{q^{k-j+2}+q-1}{q^{k-j+1}(2q-1)}, \frac{q^{k-j+1}+q-1}{q^{k-j}(2q-1)} \right)\\
            \vdots & \\
            x + \dfrac{1-q^3}{q^2(2q-1)}, & x \in \left[\dfrac{q^4+q-1}{q^3(2q-1)},\dfrac{q^3+q-1}{q^2(2q-1)} \right)\\
            x + \dfrac{1-q^2}{q(2q-1)}, & x \in \left[\dfrac{q^3+q-1}{q^2(2q-1)},\dfrac{q^2+q-1}{q(2q-1)} \right)\\
            x + \dfrac{1-q}{2q-1}, & x \in \left[\dfrac{q^2+q-1}{q(2q-1)},1 \right),\\
            \end{cases}
\end{equation*}

\noindent where $i, j \in \N$, $0 < j \leq k$ and $1 < i \leq k$. As with $T_3$ and $\mathrm{bsc}_3$, both $T_k$ and $\mathrm{bsc}_k$ depend on the parameter $q$.

\begin{thm}\label{thm:bsc} For $k > 2$, define $\mathrm{bsc}_{k} = \left(\mathrm{bsc}_{k-1} + \frac{q-1}{q^{k-1}(2q-1)}\right) \oplus \mathrm{bsc}^{\prime}_{k}$. Then $\mathrm{bsc}_{k}$ is a concatenation of two saddle connections.  
\end{thm}

\begin{proof} Clearly, $\mathrm{bsc}^{\prime}_k$ is one of the two saddle connections. We show that $\mathrm{bsc}_{k-1} + \frac{q-1}{q^{k-1}(2q-1)}$ yields a single saddle connection. 

Analogous to the previous lemma, observe that $\mathrm{bsc}^{\prime}_{k-1}$ is one of the two saddle connections in $\mathrm{bsc}_{k-1}$. We can shift $\mathrm{bsc}^{\prime}_{k-1}$ vertically by $\frac{q-1}{q^{k-1}(2q-1)}$. When we take the closure of the vertical shift, we see endpoints corresponding to $$\left(1+\frac{1}{q} + \cdots + \frac{1}{q^{k-1}},\frac{q-1}{q^{k-1}(2q-1)}\right)$$ and $$\left(1 + \frac{1}{q} + \cdots + \frac{1}{q^{k-2}},\frac{1}{q^{k-1}} + \frac{q-1}{q^{k-1}(2q-1)}\right) = \left(1 + \frac{1}{q} + \cdots + \frac{1}{q^{k-2}},\frac{2q^{2}-1}{q^{k-1}(2q-1)}\right),$$ neither of which are singular points. (See the construction of $\mathrm{bsc}_{k}^{\prime}$ in Theorem~\ref{thm:bsck'}.) The segment does not cross over the singular point as we continuously push the segment vertically, however, we see that the segment intersects that singular point at exactly one point: $$\left(1 + \frac{1}{q} + \cdots + \frac{1}{q^{k-1}}, \frac{1}{q^{k-2}(2q-1)} + \frac{q-1}{q^{k-1}(2q-1)}\right) = \left(1 + \frac{1}{q} + \cdots + \frac{1}{q^{k-1}}, \frac{1}{q^{k-1}}\right).$$ 

Further note that $\left(1 + \frac{1}{q} + \cdots + \frac{1}{q^{k-2}},\frac{2q^2-1}{q^{k-1}(2q-1)}\right)$ is identified to a point on the section $x=0$ since $$\left(1 + \frac{1}{q} + \cdots + \frac{1}{q^{k-2}},\frac{2q^2-1}{q^{k-1}(2q-1)}\right) \sim \left(0, \frac{2q^2-1}{q^{k-1}(2q-1)}\right).$$ Similarly, the singular point $\left(1 + \frac{1}{q} + \cdots + \frac{1}{q^{k-1}}, \frac{1}{q^{k-1}}\right) \sim \left(0,\frac{1}{q^{k-1}}\right)$. We will show that the linear flow connects $\left(0, \frac{2q^2-1}{q^{k-1}(2q-1)}\right)$ to $(0,\frac{1}{q^{k-1}})$ using $T_k$, and in doing so, we will see that all iterates of the map avoid the singular point. 

\begin{cl} For any $q \in \N \setminus\{0\}$, any $2 \leq n \leq q$, and any $j \leq k$, $$ T_k^{2n-4}\left(\dfrac{2q^{j-1}-1}{q^{k-1}(2q-1)}\right) = \dfrac{nq^{j-1} - 1}{q^{k-1}(2q-1)}.$$
\end{cl}

\noindent The proof of the claim follows from induction, where $q=2$ is our (trivial) base case. Indeed, 

\begin{equation*}
T_k^{2n-4}\left(\dfrac{2q^{j-1}-1}{q^{k-1}(2q-1)}\right) = T_k^2 \circ T_k^{2(n-1)-4}\left(\dfrac{2q^{j-1}-1}{q^{k-1}(2q-1)}\right) = T_k^2\left(\dfrac{(n-1)q^{j-1}-1}{q^{k-1}(2q-1)}\right)
\end{equation*}

\noindent by the inductive hypothesis. Observe that for any $j \leq k$ and $n \leq q$, $$\dfrac{(n-1)q^{j-1}-1}{q^{k-1}(2q-1)} \in \left[0, \frac{q-1}{2q-1}\right)$$ since $(n-1)q^{j-1} - 1 < q^k-q^{k-1}$ for any $j \leq k$ and $n \leq q$. Consequently,  

\begin{equation*}
T_k^2\left(\dfrac{(n-1)q^{j-1}-1}{q^{k-1}(2q-1)}\right) = T_k\left(\dfrac{(n-1)q^{j-1}-1}{q^{k-1}(2q-1)} + \frac{q}{2q-1} \right) = T_k\left( \dfrac{q^k + (n-1)q^{j-1}-1}{q^{k-1}(2q-1)}\right).
\end{equation*}

\noindent Now, observe that for any $2 \leq n \leq q$, $$\dfrac{q^k+(n-1)q^{j-1}-1}{q^{k-1}(2q-1)} \in \left[\dfrac{q^{k-j+2}+q-1}{q^{k-j+1}(2q-1)}, \dfrac{q^{k-j+1} + q - 1}{q^{k-j}(2q-1)} \right).$$ Indeed, $$\dfrac{q^{k-j+2}+q-1}{q^{k-j+1}(2q-1)} \leq \dfrac{q^k+(n-1)q^{j-1}-1}{q^{k-1}(2q-1)} < \dfrac{q^{k-j+1} + q - 1}{q^{k-j}(2q-1)}$$ since, first, $q^k + q^{j-1} - 1 \leq q^k+(n-1)q^{j-1}-1$ for any $n \geq 2$, and second, $$q^k+(n-1)q^{j-1}-1 < q^k + q^j - 1 $$ for $n \leq q$.

Consequently, 

\begin{equation*}
T_k\left( \dfrac{q^k + (n-1)q^{j-1}-1}{q^{k-1}(2q-1)}\right) = \dfrac{q^k + (n-1)q^{j-1}-1}{q^{k-1}(2q-1)} + \dfrac{1-q^{k-j+1}}{q^{k-j}(2q-1)} = \dfrac{n q^{j-1} - 1}{q^{k-1}(2q-1)},
\end{equation*}
\noindent as desired. This completes the proof of the claim. 

Now, take $n=q$ and we have $$T_k^{2q-4}\left(\dfrac{2q^2-1}{q^{k-1}(2q-1)}\right) = \dfrac{q^3 - 1}{q^{k-1}(2q-1)}.$$ Observe that $\dfrac{q^j - 1}{q^{k-1}(2q-1)} \in \left[0,\dfrac{q - 1}{2q-1}\right)$ since $q^j -1 < q^k - 1$ for any $j \leq k$. Apply $T_k$ again, and we see $$T_k\left(\dfrac{q^{j-1} - 1}{q^{k-1}(2q-1)}\right) = \dfrac{q^{j-1} - 1}{q^{k-1}(2q-1)} + \dfrac{q}{2q-1} = \dfrac{q^k + q^j - 1}{q^{k-1}(2q-1)}.$$ Iterating again, we have that $$\dfrac{q^k + q^j - 1}{q^{k-1}(2q-1)} \in \left[ \dfrac{q^{k-j+1}+q-1}{q^{k-j}(2q-1)}, \dfrac{q^{k-j}+q-1}{q^{k-j-1}(2q-1)} \right)$$ since $\dfrac{q^k + q^j - 1}{q^{k-1}(2q-1)} = \dfrac{q^{k-j+1}+q-1}{q^{k-j}(2q-1)}$ and consequently $$T_k\left(\dfrac{q^k + q^j - 1}{q^{k-1}(2q-1)}\right) = \dfrac{q^k + q^j - 1}{q^{k-1}(2q-1)} + \dfrac{1-q^{k-j}}{q^{k-j-1}(2q-1)} = \dfrac{2q^j - 1}{q^{k-1}(2q-1)}.$$ 

This provides a road map which will complete the proof. We have: 

\begin{align*}
T_k^{2q-2}\left(\dfrac{2q^{j-1} - 1}{q^{k-1}(2q-1)}\right) &= T_k^2 \circ T_k^{2q-4}\left(\dfrac{2q^{j-1} - 1}{q^{k-1}(2q-1)}\right) \\
&= T_k^2\left( \dfrac{q^{j} - 1}{q^{k-1}(2q-1)} \right) \\
&= \dfrac{2q^{j} - 1}{q^{k-1}(2q-1)}. \\
\end{align*}

\noindent Consequently, 

\begin{equation*}
T_k^{(k-3)(2q-2)}\left(\dfrac{2q^{2} - 1}{q^{k-1}(2q-1)}\right) = \dfrac{2q^{k-1} - 1}{q^{k-1}(2q-1)}
\end{equation*}

\noindent and 

\begin{equation*}
T_k^{2q-4}\left(\dfrac{2q^{k-1} - 1}{q^{k-1}(2q-1)}\right) = \dfrac{q^k - 1}{q^{k-1}(2q-1)}.
\end{equation*}

\noindent We apply $T_k$ one last time, observing that $$\dfrac{q^k - 1}{q^{k-1}(2q-1)} \in \left( \dfrac{q^{k-1}-1}{q^{k-2}(2q-1)}, \dfrac{q^{k}-1}{q^{k-1}(2q-1)} \right],$$ and we see 

\begin{equation*}
T_k\left(\dfrac{q^k - 1}{q^{k-1}(2q-1)}\right) = \dfrac{q^k - 1}{q^{k-1}(2q-1)} + \dfrac{2-q^{k-1}}{q^{k-2}(2q-1)} = \dfrac{1}{q^{k-1}},
\end{equation*}

\noindent as desired. In summary, 

\begin{equation*}
T_k^{(k-3)(2q-2)+2q-3}\left(\dfrac{2q^{2} - 1}{q^{k-1}(2q-1)}\right) = \dfrac{1}{q^{k-1}}.
\end{equation*}
\end{proof}

We record the intersections of $\mathrm{bsc}_k$ with the section $x=0$ for arbitrary $k$, using the construction of the saddle connections above. 

\begin{thm}\label{thm:bsck} The saddle connections $\mathrm{bsc}_k$ intersects $\{(0,y) : 0 < y < 1\}$ at points where $y$ is an element of

$$\begin{array}{rl}
     & \left\{\dfrac{1}{2q-1}, \cdots, \widehat{\dfrac{q}{2q-1}}, \cdots, \dfrac{2q-2}{2q-1}\right\} + \sum\limits_{i=1}^{k-1} \dfrac{q-1}{q^i(2q-1)}\\
     \bigcup & \left\{\dfrac{i}{q(2q-1)}, \dfrac{i+q^2}{q(2q-1)}\right\}_{i=1}^{q-1} + \sum\limits_{i=2}^{k-1} \dfrac{q-1}{q^i(2q-1)}\\
     & \vdots \\
     \bigcup & \left\{\dfrac{i}{q^{j-1}(2q-1)}, \dfrac{i+q^j}{q^{j-1}(2q-1)}\right\}_{i=1}^{q-1} + \sum\limits_{i=j}^{k-1} \dfrac{q-1}{q^i(2q-1)}\\
     & \vdots \\
     \bigcup & \left\{\dfrac{i}{q^{k-1}(2q-1)}, \dfrac{i+q^k}{q^{k-1}(2q-1)}\right\}_{i=1}^{q-1}
     \end{array}$$ where ``set $+$ number'' is a set where the number is added to every element in the set.
\end{thm}

\begin{proof} The saddle connections $\mathrm{bsc}_k$ consists of vertical shifts of each $\mathrm{bsc}_j^{\prime}$ for every $1 < j < k$ (where the vertical shift depends on $j$), as well as $\mathrm{bsc}_k^{\prime}$ and a vertical shift of $\mathrm{bsc}_1$. Thus, the intersections of $\mathrm{bsc}_k$ correspond to the intersections of $\mathrm{bsc}_1$ shifted up by $\sum_{i=1}^{k-1} \frac{q-1}{q^i(2q-1)}$, $\mathrm{bsc}_2^{\prime}$ shifted up by $\sum_{i=2}^{k-1} \frac{q-1}{q^i(2q-1)}$, and so on, until we reach $\mathrm{bsc}_k^{\prime}$.
\end{proof}

%%%%%%%%%%%%%%%%%%%%%%%%%%%%%%%%%%%%%
%%%%%%%%%%%%%%%%%%%%%%%%%%%%%%%%%%%%%
%%%%%%%%%%%%%%%%%%%%%%%%%%%%%%%%%%%%%
%%%%%%%%%%%%%%%%%%%%%%%%%%%%%%%%%%%%%
%%%%%%%%%%%%%%%%%%%%%%%%%%%%%%%%%%%%%
%%%%%%%%%%%%%%%%%%%%%%%%%%%%%%%%%%%%%
%%%%%%%%%%%%%%%%%%%%%%%%%%%%%%%%%%%%%
%%%%%%%%%%%%%%%%%%%%%%%%%%%%%%%%%%%%%

\medskip

\section{Infinitely many cylinders (or finding successive cylinders in the $\frac{1}{2-r}$-direction)}\label{sec:cyl}

In this section, we define a map $f_r$ to construct new cylinders from existing cylinders. 

\begin{defn}\label{defn:f_r} For $0<r<1,$ we define $\tilde{f}_r: \R^2 \to \R^2$, where $$\tilde{f}_r:\begin{pmatrix} x \\ y\end{pmatrix} \mapsto \begin{pmatrix} r x + 1 \\ r y\end{pmatrix} = r \begin{pmatrix}
    x\\
    y
\end{pmatrix}+\begin{pmatrix}1\\0   
\end{pmatrix}.$$ Observe that $\tilde{f}_r$ is an injective map. \end{defn}

The map does not descend from $\R^2$ to a well-defined map on the quotient of the polygonal representation of an armadillo tail. The issue is that the map does not respect the vertical gluings (by horizontal translations). For instance, the left edge of $\square_1$ is mapped to the left edge of $\square_2$, but the left edge of $\square_2$ is glued to the \emph{right} edge of $\square_1$. However, the map \emph{does} descend to a partial quotient, where we only identify the top and bottom edges, since $\tilde{f}_r$ respects the identifications along the tops and bottoms of the squares. That is the content of the following lemma, whose proof is elementary.

\begin{lem}
Let $P_r$ be the polygonal representation of an armadillo tail $X$ with parameter $r$ such that the polygonal representation is embedded in $\R^2$ and the edge identifications forgotten. Let $X^{tb}$ be a quotient of $P_r$ by identifying the top and bottom edges only. Let $\square_{k}^{tb}$ denote the $k$th-square in $X^{tb}$. Then $\tilde{f}_r$ descends to a map $f_r$ on $X^{tb}$. The image of $\square_{k}^{tb}$ under $f_r$ is $\square_{k+1}^{tb}.$
\end{lem}

Let $q: P_r \to X^{tb}$ be the quotient map identifying the top and bottom edges of the polygon. Let $\mathrm{cyl}_k$ be a cylinder, and lift $\mathrm{cyl}_k$ to $X^{tb}$. Call this lift $\mathrm{cyl}_k^{tb}.$
Let $L$ and $R$ denote the unidentified left and right edges of the polygon. Inductively define $\mathrm{cyl}_{k+1}$ as the closure of $q \circ f_r (\mathrm{cyl}_k^{tb} \setminus (L \cup R))$ with respect to the linear flow in the $\frac{1}{2-r}$-direction. Observe that $\mathrm{cyl}_{k+1}$ does not depend on the chosen lift of $\mathrm{cyl}_k$.

Given a geometric armadillo tail with parameter $r=1/q,$ first we show that $\mathrm{cyl}_k$ lies entirely on $X_{k+1}=\bigcup\limits_{i=1}^{k+1} \square_i.$ Then $q \circ f_r(\mathrm{cyl}_k)$ is a subset of a cylinder that lies in $X_{k+1}^{tb}\setminus \square_1.$ We call this a \emph{partial cylinder}. We will show that there is a circle rotation on $\{0\} \times [0,1]$ that fills in $q \circ f_r(\mathrm{cyl}_k)$ at the points of discontinuity.

We define where the circle rotation is defined, and prove that waist curves of cylinders are periodic points under the circle rotation.

We define the \emph{generation zone} as the preimage of $\mathrm{cyl}_1$ in $X \setminus \square_1$ under $f_r$:

\begin{align*}
\text{generation zone} & := f^{-1}_{r}( \mathrm{Int}(\mathrm{cyl}_1 \cap \square_2)) \\
&= \left\{(x,y): \dfrac{1}{2-r} x < y < \dfrac{1}{2-r} x + \dfrac{1-r}{2-r}, \, 0 \leq x \leq 1\right\}. \\
\end{align*} 

Given the set of points where $\mathrm{cyl}_k$ intersects $\{0\}\times[0,1]$ and $\{1\}\times[0,1],$ we remove the points that lie in the generation zone. Define sets $S_1$ (and $S_2,$ resp.) on $\{0\}\times[0,1]$ (and $\{1\}\times[0,1],$ resp.) as the image of the remaining points under $f_r.$ That is, $$S_1= \text{proj}_y\circ f_r\left(\gamma \cap \left\{(0,y): 0 < y < \dfrac{1}{2-r}\right\}\right)$$ and $$S_2 = \left\{f_r \left(\gamma \cap \left\{(0,y): \dfrac{1-r}{2-r} < y < 1\right\}\right)\right\}$$ where $\text{proj}_y(x,y) = (0,y)$ is the projection onto the $y$-axis. 

Recall the circle rotation $T(x) = x + \frac{q}{2q-1}$. Note that $q\circ f_r(\mathrm{cyl}_k) \subset X^{tb}\setminus\square_1$ is a subset of a cylinder. Since the circle rotation is a section of the linear flow, we ``fill in'' $q\circ f_r(\mathrm{cyl}_k)$ at the points of discontinuity to construct $\mathrm{cyl}_{k+1}.$ Figure~\ref{fig:fr-of-cyl1} illustrates the setting.  

\begin{figure}[h]
    \centering    
    \includegraphics[width=0.9\linewidth]{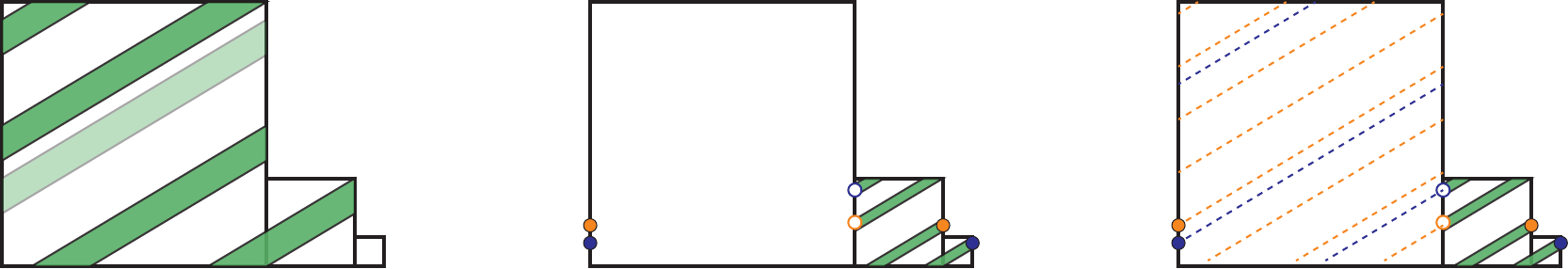}
    \caption{$\mathrm{cyl}_1 \setminus \text{generation zone} (\text{left}),$ $f_r\left(\mathrm{cyl}_1 \setminus \text{generation zone}\right)$ (center), connecting $S_1$ and $S_2$ via the circle rotation $T$ (right). The dotted lines represent $\mathrm{bsc}_2$ in $\square_1.$}
    \label{fig:fr-of-cyl1}
\end{figure}

We illustrate the simplest case ($k=1$) before we prove the general case for all $k.$ Since we can explicitly write the points at which $\mathrm{bsc}_1$ intersects the $y$-axis (Theorem~\ref{thm:cyl1}) we list them here and take a waist curve $\gamma$ to be $\epsilon$ above these points where $0 < \epsilon < \frac{q-1}{q(2q-1)}$. The $y$-coordinate of these points are $$\left\{\frac{1}{2q-1}, \ldots, \widehat{\frac{q}{2q-1}}, \ldots,\frac{2q-2}{2q-1}\right\},$$ hence we have $$S_1 = \left\{(0,y): y= \frac{1}{q(2q-1)}, \ldots, \frac{q-1}{q(2q-1)}\right\}, \quad S_2 = \left\{(1,y): y = \frac{q-1}{q(2q-1)}, \ldots, \frac{2q-2}{q(2q-1)}\right\}.$$
Take $i=1, \ldots, q-2,$ then $$\frac{i}{q(2q-1)} \xrightarrow{T}  \frac{i+q^2}{q(2q-1)} \xrightarrow{T}  \frac{i+2q^2}{q(2q-1)} \equiv  \frac{i}{q(2q-1)},$$ where $\frac{1}{q} <  \frac{i+q^2}{q(2q-1)} < 1$ for $i \in \{1, \ldots, q-2\}.$ We note that $\left\{\left(0,\frac{i+q^2}{q(2q-1)}\right)\right\}_{i=1}^{q-2}$ are $q-2$ additional points where we hit the $y$-axis.

When $i=q-1,$ $$ \frac{q-1}{q(2q-1)} \xrightarrow{T} \frac{q-1+q^2}{q(2q-1)} \xrightarrow{T} \frac{q-1+2q^2}{q(2q-1)} \equiv \frac{2q-1}{q(2q-1)}.$$ First note that $\left(1,\frac{2q-1}{q(2q-1)}\right)$ is a singularity. Since $\gamma$ is slightly above this, we continue by iterating $T$. Moreover, note that $T^{2q-1}\left(\frac{q-1}{q(2q-1)}\right) = \frac{q-1+(2q-1)q^2}{q(2q-1)} \equiv \frac{q-1}{q(2q-1)}.$ We will show below that for any $m<2q-1,$ the $m$th iterate hits the portal of $\square_1,$ i.e., $\frac{1}{q} < T^m(\frac{q-1}{q(2q-1)}) = \frac{q-1+mq^2}{q(2q-1)} < 1.$ This adds an additional $q-2$ points where we hit the $y$-axis.

Given $\mathrm{cyl}_k,$ we define the \emph{skew-width of $\mathrm{cyl}_k$} as the vertical distance (vertical in relation to the polygonal representation) between the two boundary saddle connections of $\mathrm{cyl}_k$ through the cylinder.

\begin{lem}\label{lem:skewwidth} Given a geometric armadillo tail with parameter $1/q$, the skew-width of $\mathrm{cyl}_k$ is $\frac{q-1}{q^k (2q-1)}$, and the width of $\mathrm{cyl}_k$ is $\frac{q-1}{q^k \sqrt{q^2 + (2q-1)^2}}$.\end{lem}

\begin{thm} The circle rotation $T:[0,1]/_\sim \to [0,1]/_\sim$ where $T(x) = x + \frac{q}{2q-1}$ maps the $y$-coordinates in $S_1$ to the $y$-coordinates in $S_2$ defined above. \end{thm}

\begin{proof} Theorem~\ref{thm:bsck} tells us exactly where $\mathrm{bsc}_k$ intersects $\{[0,y] : 0 < y < 1\}.$ Recall that $S_1$ consists of points on $\{0\}\times[0,1]$ whose $y$-coordinates are $$\dfrac{1}{q}\left(\bigcup\limits_{j=1}^{k-1} \left\{\dfrac{i}{q^{j-1}(2q-1)}\right\}_{i=1}^{q-1} + \sum\limits_{i=j}^{k-1} \dfrac{q-1}{q^i(2q-1)} \cup \left\{\dfrac{i}{q^{k-1}(2q-1)}\right\}_{i=1}^{q-1}\right)$$ and $S_2$ consists of points on $\{1\}\times[0,1]$ whose $y$-coordinates are 
$$\begin{array}{rl} & \dfrac{1}{q} \left(\left\{\dfrac{q-1}{2q-1}, \dfrac{q+1}{2q-1}, \cdots, \dfrac{2q-2}{2q-1}\right\} + \sum\limits_{i=1}^{k-1} \dfrac{q-1}{q^i(2q-1)}\right) \\
\cup & \bigcup\limits_{j=2}^{k-1} \dfrac{1}{q} \left(\left\{\dfrac{i+q^j}{q^{j-1}(2q-1)}\right\}_{i=1}^{q-1} + \sum\limits_{i=j}^{k-1} \dfrac{q-1}{q^i(2q-1)}\right)\\
\cup & \dfrac{1}{q} \left\{\dfrac{i+q^k}{q^{k-1}(2q-1)}\right\}_{i=1}^{q-1}.\end{array}$$ 
These can be simplified to $$S_1 = \left\{(0,y): y \in \bigcup\limits_{j=1}^{k-1}\left\{\dfrac{(i+1)q^{k-j}-1}{q^k(2q-1)}\right\}_{i=1}^{q-1} \cup \left\{\dfrac{i}{q^k(2q-1)}\right\}_{i=1}^{q-1}\right\}$$ and 
$$S_2 = \left\{(1,y):y \in \left\{\dfrac{(1+i)q^{k-1}+q^k-1}{q^k(2q-1)}\right\}_{i=-1,i\neq0}^{q-2} \cup \bigcup\limits_{j=2}^{k-1}\left\{ \dfrac{(1+i)q^{k-j}+q^k-1}{q^k(2q-1)}\right\}_{i=1}^{q-1}\cup \left\{\dfrac{i+q^k}{q^{k-1}(2q-1)}\right\}_{i=1}^{q-1}\right\}.$$

We will show that under some iterate of $T,$ $S_1$ maps onto $S_2.$ We will break this down into three cases. In each case, we show that a point in $S_1$ maps to a point in $S_2$ under $T^2$ (or $T^{2q-1}$) but any fewer iterate of $T$ maps it to the complement of $S_2$ on $\{1\}\times[0,1],$ i.e., $\left\{(1,y): \frac{q}{2q-1} < y < 1\right\}.$

\textbf{Case 1.} First we have $T\left(\frac{i}{q^k(2q-1)}\right) = \frac{i+q^k}{q^k(2q-1)}$ for any $i \in \{1, \ldots, q-1\}.$

\textbf{Case 2-1.} Consider the cases where $j=2, \ldots, k-1,$ and $i=1,\ldots, q-1.$ Then $$T^2\left(\dfrac{(i+1)q^{k-j}-1}{q^k(2q-1)}\right)=\dfrac{(i+1)q^{k-j}-1+2q^{k+1}}{q^k(2q-1)} \equiv \dfrac{(i+1)q^{k-j}-1+q^k}{q^k(2q-1)}.$$ We show that $T\left(\frac{(i+1)q^{k-j}-1}{q^k(2q-1)}\right)$ does not hit any point in $S_2,$ i.e., $$\dfrac{1}{q} < \dfrac{(i+1)q^{k-j}-1+q^{k+1}}{q^k(2q-1)} < 1.$$ The left inequality holds since it is equivalent to the inequalities below: $$\begin{array}{rl} q^{k-1}(2q-1) & < (i+1) q^{k-j} - 1 + q^{k+1}\\
1 & < (i+1) q^{k-j} + q^{k+1} - q^{k-1} (2q-1)\\
& = q^{k-1}\left((i+1)q^{1-j}+(q-1)^2\right)\end{array},$$ and the right inequality holds since it is equivalent to $$\begin{array}{rl} q^k(q-1) & > (i+q)q^{k-j}-1\\
q^k \left(q-1-(i+1)q^{-j}\right) + 1 & \geq q^k\left(q-1-\frac{q}{q^j}\right) + 1 > 0.\end{array}$$

\textbf{Case 2-2.} Next, the cases where $j=1$ and $i=1,\ldots,q-2$ can be shown with the same technique as in Case 2-1: we have $$T^2\left(\dfrac{(i+1)q^{k-1}-1}{q^k(2q-1)}\right) = \dfrac{(i+1)q^{k-1}-1+2q^k}{q^k(2q-1)}\equiv \dfrac{(i+1)q^{k-1}-1+q^k}{q^k(2q-1)},$$ and $\frac{1}{q} < T\left(\frac{(i+1)q^{k-1}-1}{q^k(2q-1)}\right) < 1.$ The left inequality holds since $$\begin{array}{rl} q^{k-1}(2q-1) & < (i+1) q^{k-1} - 1 + q^{k+1}\\
1 & < q^{k-1}\left(i+2+2q+q^2\right).\end{array}$$ However, the right inequality holds only for $i=1,\ldots, q-2:$
$$\begin{array}{rl} (i+1) q^{k-1} - 1 + q^{k+1} & < 2q^{k+1}-q^k\\
-1 & < q^{k+1}-(i+2)q^k = q^k \left(q-(i+2)\right).\end{array}$$

\textbf{Case 3.} Lastly, we deal with $j=1$ and $i=q-1.$

After $2q-1$-iterates, $\frac{q^k-1}{q^k(2q-1)}$ is mapped to itself. We need to show that for any $m < 2q-1,$ $T^m\left(\frac{q^k-1}{q^k(2q-1)}\right)$ falls between $\frac{1}{q}$ and 1, hence does not hit any other point in $S_2.$

If $m=2l,$ $(l=1, \ldots, q-1),$ then $$T^m\left(\dfrac{q^k-1}{q^k(2q-1)}\right)=\dfrac{q^k-1+2lq^k}{q^k(2q-1)}\equiv\dfrac{(l+1)q^k-1}{q^k(2q-1)}.$$ We show that $$\frac{1}{q} < \dfrac{(l+1)q^k-1}{q^k(2q-1)} < 1.$$ The left-hand inequality is equivalent to $$\begin{array}{rl} q^{k-1}(2q-1) & < (l+1) q^k - 1\\
1 & < (l+1) q^k - q^{k-1} (2q-1) = q^{k-1} \left((l+1)q-(2q-1)\right) = \left((l-1)q+1\right),\end{array}$$ and the right-hand inequality is equivalent to $$\begin{array}{rl} (l+1)q^k-1 & < q^k (2q-1)\\
-1 & < q^k (2q-2-l).\end{array}$$

If $m=2l+1,$ $(l=1, \ldots, q-2),$ then $$T^m\left(\dfrac{q^k-1}{q^k(2q-1)}\right) = \dfrac{q^k-1+(2l+1)q^{k+1}}{q^k(2q-1)} \equiv \dfrac{(l+1)q^k - 1 + q^{k+1}}{q^k(2q-1)}.$$ Again, we show that this does not hit any point in $S_2.$ First it is greater than $\frac{1}{q}$ since $$\begin{array}{rl} q^{k-1}(2q-1) & < (l+1)q^k - 1 + q^{k+1}\\
1 & < q^{k-1} \left((l+1)q + q^2 + (1-2q)\right) = q^{k-1} \left(q^2 + (l-1)q+1\right),\end{array}$$ and less than 1 since $$\begin{array}{rl}(l+1)q^k - 1 + q^{k+1} & < q^k(2q-1)\\
-1 & < q^k \left(2q-1-(l+1)-q\right) = q^k \left(q-(l+2)\right).\end{array}$$ Take $\gamma$ to be $\varepsilon$ above $\mathrm{bsc}_k$ where $0 < \varepsilon < \frac{q-1}{q^k(2q-1)}.$ Thus, we have connected the disconnected segments of $f_r(\mathrm{cyl}_k)$ to construct a waist curve of $\mathrm{cyl}_{k+1}.$\end{proof}

Figure~\ref{fig:cyl} shows the first few cylinders in this cylinder decomposition for $r = \frac{1}{2}$.

\begin{figure}[htbp]
	\begin{center}
	\includegraphics[width=0.7\linewidth]{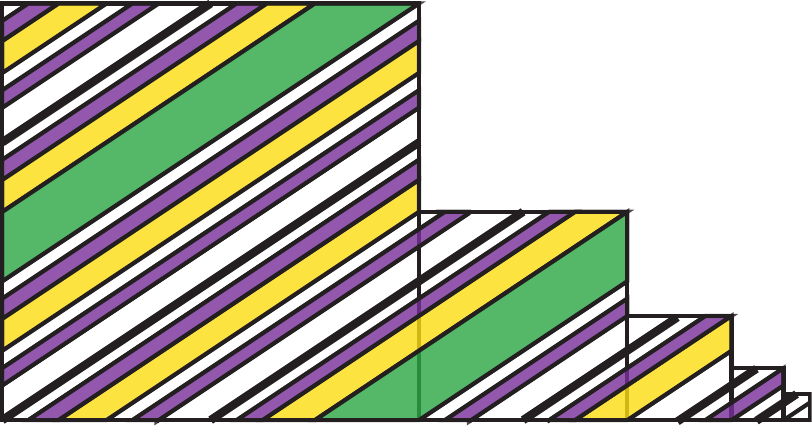}
	\caption{Cylinder decomposition on the geometric armadillo tail ($r=1/2$)}
 \label{fig:cyl}
 \end{center}
 
\end{figure}

In the following theorem, we summarize the work done in this section by identifying our construction as a cylinder decomposition. In other words, we show that the closure of the cylinder decomposition is the entire surface.

\begin{thm}\label{thm:main} Given a geometric armadillo tail with parameter $r = \frac{1}{q},$ there exists a cylinder decomposition in the $\frac{1}{2-r}$-direction with a rigid spine. \end{thm}

\begin{proof} We first restrict to $\square_1,$ then use $f_r$ to extend our results to the remaining squares.

For $k>1,$ $\mathrm{bsc}_k'$ hits the $y$-axis $2(q-1)$ many times at points where the $y$-coordinate is $\frac{i}{q^{k-1}(2q-1)}$ and $\frac{i+q^k}{q^{k-1}(2q-1)}$, for $i = 1, \ldots, q-1$. (Theorem~\ref{thm:bsck'}). This corresponds to the times $\mathrm{cyl}_k$ wraps around $\square_1$ when it is not adjacent to $\mathrm{cyl}_{k-1}$. Since there is always a part of $\mathrm{cyl}_{k+1}$ that lies directly above $\mathrm{cyl}_k,$ we sum the skew-widths of subsequent cylinders to observe the convergence in $\square_1.$ We have $$\frac{i}{q^{k-1}(2q-1)} + \frac{q-1}{q^{k}(2q-1)} + \frac{q-1}{q^{k+1}(2q-1)} + \cdots = \frac{i+1}{q^{k-1}(2q-1)}.$$ The same holds for $\frac{i+q^k}{q^{k-1}(2q-1)}$. In other words, for $i=1, \ldots , q-2,$ the part of the cylinders in $\square_1$ converge to the bottom of a part of $\mathrm{cyl}_k.$ Take $i=q-1,$ then the same computation for $\frac{q-1}{q^{k-1}(2q-1)}$ yields $\frac{1}{q^{k-2}(2q-1)}$ which is a part of the bottom of $\mathrm{cyl}_{k-1},$ and for $\frac{q-1+q^k}{q^{k-1}(2q-1)},$ we get $\frac{1+q^{k-1}}{q^{k-2}(2q-1)},$ which is also a part of the bottom of $\mathrm{cyl}_{k-1}.$ In other words, given $\mathrm{cyl}_k \cap \square_1$ not adjacent to $\mathrm{cyl}_{k-1},$ there is a sequence $\mathrm{cyl}_i \cap \square_1$ that converges to the bottom of $\mathrm{cyl}_k \cap \square_1$ or $\mathrm{cyl}_{k-1} \cap \square_1.$  

When $k=1,$ we refer to Theorem~\ref{thm:cyl1}. The points at which $\mathrm{bsc}_1$ intersects the $y$-axis is at points where the $y$-coordinate is $\frac{i}{2q-1}$ for $i=1, \ldots, 2q-2$ but not $i=q.$ Then adding the sum of skew-widths $$\frac{i}{2q-1} + \frac{q-1}{q(2q-1)} + \frac{q-1}{q^2(2q-1)} + \cdots = \frac{i+1}{2q-1}$$ shows us that $\mathrm{cyl}_1 \cap \square_1$ converges to either the bottom of $\mathrm{cyl}_1 \cap \square_1$, or a singularity (when $i=2q-2$), or the rigid spine (when $i=q-1$).

Similar computations hold for $\square_2$, and applying the $f_r$ map implies that the same convergence applies in all squares. 

\end{proof}

%%%%%%%%%%%%%%%%%%%%%%%%%%%%%%%%%%%%%
%%%%%%%%%%%%%%%%%%%%%%%%%%%%%%%%%%%%%
%%%%%%%%%%%%%%%%%%%%%%%%%%%%%%%%%%%%%
%%%%%%%%%%%%%%%%%%%%%%%%%%%%%%%%%%%%%
%%%%%%%%%%%%%%%%%%%%%%%%%%%%%%%%%%%%%
%%%%%%%%%%%%%%%%%%%%%%%%%%%%%%%%%%%%%
%%%%%%%%%%%%%%%%%%%%%%%%%%%%%%%%%%%%%
%%%%%%%%%%%%%%%%%%%%%%%%%%%%%%%%%%%%%

\medskip

\appendix

\section{Veech group of an armadillo tail}\label{app:veech-group}

Let $X$ be an armadillo tail. Let $\Aff^+(X)$ be the group of orientation-preserving affine diffeomorphisms, and $\Gamma(X)$ its Veech group. See \cite{DeLaCroix-Hubert-Valdez} for definitions. The following proof was suggested by Anja Randecker. 

\begin{thm}\label{thm:affine-diffeo}
Let $X$ be an armadillo tail with the property that for any $\varepsilon >0,$ there exists a square $\square_k$ in the polygonal representation with side length $l_k < \varepsilon.$ Then $\Aff^+(X)$ is isomorphic to $\mathbb{Z}.$
\end{thm}

\begin{proof} Assume that $\phi \in \Aff^+(X)$ does not fix the left edge of $\square_1.$ Then it must send the concatenation of infinitely many saddle connections along the left edge of $\square_1$ to a distinct concatenation of infinitely many saddle connections, all of which are in the same direction and whose sum of lengths is finite. Consequently, the image of the left edge of $\square_1$ must have arbitrarily short saddle connections. In the polygonal representation of the surface, we see that the only sequences of arbitrarily short saddle connections either lie along the left edge of $\square_1,$ or appear in infinitely many (distinct) squares $\square_k.$ Indeed, for a saddle connection of length $\delta,$ the saddle connection must either be in $\square_1$ or it must intersect some square whose side length is less than $\delta.$

Moreover, $\square_1$ is mapped under $\phi$ to a parallelogram with positive dimensions, where the interior of the parallelogram has no occurrences of the wild singularity. Since the saddle connections along one edge of this parallelogram must be arbitrarily short, the parallelogram must intersect squares with arbitrarily short side lengths. However, the parallelogram has positive dimensions, hence, it cannot pass through arbitrarily small squares, a contradiction. 

Thus, any $\phi \in \Aff^+(X)$ must fix the left side of $\square_1.$ Given that each square has modulus 1, $\Aff^+(X)$ is generated by some $\psi$ such that $D\psi = \begin{pmatrix} 1 & 1 \\ 0 & 1 \end{pmatrix}.$
\end{proof}

\begin{cor} Let $X$ be as in Theorem \ref{thm:affine-diffeo}. Then the Veech group $\Gamma(X)$ is isomorphic to $\Z.$
\end{cor}

%%%%%%%%%%%%%%%%%%%%%%%%%%%%%%%%%%%%%
%%%%%%%%%%%%%%%%%%%%%%%%%%%%%%%%%%%%%
%%%%%%%%%%%%%%%%%%%%%%%%%%%%%%%%%%%%%
%%%%%%%%%%%%%%%%%%%%%%%%%%%%%%%%%%%%%
%%%%%%%%%%%%%%%%%%%%%%%%%%%%%%%%%%%%%
%%%%%%%%%%%%%%%%%%%%%%%%%%%%%%%%%%%%%
%%%%%%%%%%%%%%%%%%%%%%%%%%%%%%%%%%%%%
%%%%%%%%%%%%%%%%%%%%%%%%%%%%%%%%%%%%%

\medskip

\section{Measurements of cylinders}\label{app:area-of-cyls}

In this section, we compute and record the length of a waist curve, the modulus, and the area for each cylinder. We define the \emph{modulus} of a cylinder as the ratio $\frac{\text{circumference}}{\text{width}},$ and note that it may be the readers preference to take the reciprocal of the numbers computed below. The computations in Proposition \ref{prop:skew-width-sum} and \ref{prop:area-sum} that follow the measurements provide evidence for the validity our work. 

We start by computing the horizontal displacement of each waist curve. The table below lists the side lengths of each square and the number of times a waist curve of $\mathrm{cyl}_k$ goes through each square. This follows from the circle rotation defined in Section \ref{sec:cyl}. 

$$\begin{tabular}{c|c|c|c|c|c|c|c|}
     & $\square_1$ & $\square_2$ & $\cdots$ & $\square_i$ & $\cdots$ & $\square_k$ & $\square_{k+1}$ \\ \hline
\text{side length} & 1 & $1/q$ & & $1/q^{i-1}$ & & $1/q^{k-1}$ & $1/q^k$ \\     \hline
\# & $k(2q-2)-1$ & $(k-1)(q-1)$ & & $(k-i+1)(q-1)$ & & $q-1$ & $1$     \end{tabular}$$

The horizontal displacement of a waist curve of $\mathrm{cyl}_k$ is $$\begin{array}{rl} & k(2q-2)-1 + \dfrac{1}{q} (k-1)(q-1) + \dfrac{1}{q^2} (k-2)(q-1) + \cdots + \dfrac{1}{q^{k-1}}(q-1) + \dfrac{1}{q^k}\\
= & k(2q-2) - 1 + \sum\limits_{i=1}^{k-1} \dfrac{(k-i)(q-1)}{q^i} + \dfrac{1}{q^k} \\
= & k(2q-2) - 1 + 
\dfrac{q-1}{q^k} \sum\limits_{i=1}^{k-1} (k-i) q^{k-i} + \dfrac{1}{q^k}\\
= & k(2q-2) - 1 + \dfrac{(k-1) q^{k+1}-kq^k+q}{q^k (q-1)} + \dfrac{1}{q^k}.\end{array}$$ The last equality follows since: $$\begin{array}{rl} \sum\limits_{i=1}^{k-i} (k-i) q^{k-i} & = q + 2q^2 + 3q^3 + \cdots + (k-1)q^{k-1}\\
& = q + 2q^2 + 3q^3 + \cdots + (k-1)q^{k-1} + \left(q + \cdots + q^{k-1}\right) - \left(q + \cdots + q^{k-1}\right)\\
& = 2q + 3q^2 + \cdots + k q^{k-1} - \dfrac{q(q^{k-1}-1)}{q-1}\\
& = \left(q^2 + \cdots + q^k\right)' - \dfrac{q^k-q}{q-1}\\
& = \left(\dfrac{q^2 (q^{k-1}-1)}{q-1}\right)' - \dfrac{q^k-q}{q-1}\\
& = \dfrac{(k-1) q^{k+1}- k q^k + q}{(q-1)^2}.\end{array}$$

\begin{prop}\label{prop:length of cyl} The horizontal displacement of the waist curve of $\mathrm{cyl}_k$ is $$(2q-1)\left(k - \dfrac{q^k - 1}{q^k (q-1)}\right)$$ and the actual length of the waist curve, i.e., the circumference of $\mathrm{cyl}_k$ is $$\left( k - \dfrac{q^k - 1}{q^k (q-1)}\right) \sqrt{(2q-1)^2 + q^2}.$$ Furthermore, the modulus of $\mathrm{cyl}_k$ is given as $$\dfrac{\text{circumference}}{\text{width}} = \dfrac{q^2 + (2q-1)^2}{(q-1)^2} \left(k q^{k+1} -(k+1) q^k + 1\right),$$ and the area of $\mathrm{cyl}_k$ is given by $$\text{area}\left(\mathrm{cyl}_k\right) = \left(k - \dfrac{q^k - 1}{q^k (q-1)}\right) \dfrac{q-1}{q^k}.$$\end{prop}

\begin{prop}\label{prop:skew-width-sum} Given a geometric armadillo tail with parameter $r = \frac{1}{q},$ $q \in \N \setminus \{1\},$ and a cylinder decomposition in the $\frac{1}{2-r}$-direction, the sum of skew-widths in $\square_1$ is 1.\end{prop}

\begin{proof}
The number of times $\mathrm{cyl}_k$ intersects the $y$-axis is $2kq -(2k+1).$ Hence the sum of all skew-widths is:
\begin{align*}
\sum_{k=1}^\infty \left(2kq-(2k+1)\right) \cdot \frac{q-1}{q^k (2q-1)} & = \frac{q-1}{2q-1} \sum_{k=1}^\infty \left(\frac{2k}{q^{k-1}} - \frac{2k}{q^k} - \frac{1}{q^k}\right)\\
& = \frac{q-1}{2q-1} \left(\frac{2q^2}{(q-1)^2} - \frac{2q}{(q-1)^2} - \frac{1}{q-1}\right) = 1
\end{align*} for all $q \in \N \setminus \{1\}.$ 
\end{proof}

Next, we verify that given $r=1/q,$ the infinite sum of $\text{area}(\mathrm{cyl}_k)$ is equal to $\frac{1}{1-r^2},$ hence there exists an infinite cylinder decomposition in the $\frac{1}{2-r}$-direction.

\begin{prop}\label{prop:area-sum} Given a geometric armadillo tail with parameter $r = \frac{1}{q},$ $q \in \N \setminus \{1\},$ and a cylinder decomposition in the $\frac{1}{2-r}$-direction, the sum of the areas of the cylinders is the area of surface.\end{prop}

\begin{proof} From Proposition~\ref{prop:length of cyl}, we write $\text{area}\left(\mathrm{cyl}_k\right) = \frac{k (q-1)}{q^k} - \frac{1}{q^k} + \frac{1}{q^{2k}}.$ Following the same spirit as a previous remark, we use $\sum\limits_{i=1}^\infty i r^i = \frac{r}{(1-r)^2},$ for $|r| < 1.$ The sum of the first terms is $$\sum\limits_{k=1}^\infty \dfrac{k(q-1)}{q^k} = \dfrac{q}{q-1}.$$ The second and third terms are geometric sequences, hence we have $$\sum\limits_{k=1}^\infty \text{area}\left(\mathrm{cyl}_k\right) = \dfrac{q}{q-1} + \dfrac{1}{q-1} + \dfrac{1}{q^2 - 1} = \dfrac{q^2}{q^2-1},$$ our desired result.\end{proof}

%References 

%\bibliographystyle{abbrv}
\bibliographystyle{abbrvnat}
\bibliography{ref}

@book{bowman,
title = {Finiteness conditions on translation surfaces},
ISBN = {978-0-8218-9029-5},
publisher = {American Mathematical Society},
author = {Joshua P. Bowman},
editor = {Jiang, Yunping and Mitra, Sudeb},
year = {2012},
pages = {31--40}
}

@article{bowman-arnouxyoccoz,
  title={The complete family of {A}rnoux--{Y}occoz surfaces},
  author={Joshua P. Bowman},
  journal={Geometriae Dedicata},
  year={2012},
  volume={164},
  pages={113--130},
  url={https://api.semanticscholar.org/CorpusID:254500030}
}

@article{ddl2,
author = {Degli Esposti, Mirko and Del Magno, Gianluigi and Lenci, Marco},
year = {1999},
month = {07},
pages = {},
title = {Escape orbits and Ergodicity in Infinite Step Billiards},
volume = {13},
journal = {Nonlinearity},
doi = {10.1088/0951-7715/13/4/316}
}

@article {ddl1,
    AUTHOR = {Degli Esposti, Mirko and Del Magno, Gianluigi and Lenci,
              Marco},
     TITLE = {An infinite step billiard},
   JOURNAL = {Nonlinearity},
  FJOURNAL = {Nonlinearity},
    VOLUME = {11},
      YEAR = {1998},
    NUMBER = {4},
     PAGES = {991--1013},
      ISSN = {0951-7715,1361-6544},
   MRCLASS = {58F11},
  MRNUMBER = {1632594},
MRREVIEWER = {Andr\'as\ Kr\'amli},
       DOI = {10.1088/0951-7715/11/4/013},
       URL = {https://doi.org/10.1088/0951-7715/11/4/013},
}

@book{randecker2016,
 author = {Randecker, Anja},
 doi = {10.5445/KSP/1000050964},
 isbn = {978-3-7315-0456-6},
 month = {Apr},
 pages = {164},
 publisher = {KIT Scientific Publishing},
 title = {Geometry and topology of wild translation surfaces},
 year = {2016}
}

@inproceedings{DeLaCroix-Hubert-Valdez,
  title={Infinite Translation Surfaces in the Wild},
  author={Vincent Delecroix and Pascal Hubert and Ferr{\'a}n Valdez},
  year={2024},
  url={https://api.semanticscholar.org/CorpusID:268297082}
}

@article{Viana,
author = {Viana, Marcelo},
journal = {Revista Matemática Complutense},
language = {eng},
number = {1},
pages = {7-100},
title = {Ergodic theory of interval exchange maps.},
url = {http://eudml.org/doc/40876},
volume = {19},
year = {2006},
}

@article{hooper-hubert-weiss,
title = {Dynamics on the infinite staircase},
journal = {Discrete and Continuous Dynamical Systems},
volume = {33},
number = {9},
pages = {4341--4347},
year = {2013},
issn = {1078-0947},
doi = {10.3934/dcds.2013.33.4341},
url = {https://www.aimsciences.org/article/id/fcbb257b-52c8-44ae-aef8-f45b2cfb3fe6},
author = {W. Patrick Hooper and Pascal Hubert and Barak Weiss}
}

@article{HS-intro,
author = {Hubert, Pascal and Schmidt, Thomas},
year = {2006},
month = {01},
pages = {},
title = {An Introduction to {V}eech Surfaces},
journal = {Handbook of Dynamical Systems, {v}ol. 1B}
}

@article{masur-closed-trajectories,
    AUTHOR = {Masur, Howard},
     TITLE = {Closed trajectories for quadratic differentials with an
              application to billiards},
   JOURNAL = {Duke Math. J.},
  FJOURNAL = {Duke Mathematical Journal},
    VOLUME = {53},
      YEAR = {1986},
    NUMBER = {2},
     PAGES = {307--314},
      ISSN = {0012-7094,1547-7398},
   MRCLASS = {30F30 (30C60 58F17)},
  MRNUMBER = {850537},
MRREVIEWER = {H.\ Renelt},
       DOI = {10.1215/S0012-7094-86-05319-6},
       URL = {https://doi-org.bris.idm.oclc.org/10.1215/S0012-7094-86-05319-6},
}

\end{document}